\documentclass{article}

\usepackage{microtype}
\usepackage{bm}
\usepackage{graphicx}
\usepackage{subfigure}
\usepackage{booktabs} %

\usepackage{hyperref}

\usepackage[accepted]{arxiv_icml2021}

\usepackage{macros}

\icmltitlerunning{Smooth $\bm{\Wp}$: Structure, Empirical Approximation, and Applications}

\begin{document}

\twocolumn[
\icmltitle{Smooth \textbf{\emph{p}}-Wasserstein Distance:\\ Structure, Empirical Approximation, and Statistical Applications}

\icmlsetsymbol{equal}{*}

\begin{icmlauthorlist}
\icmlauthor{Sloan Nietert}{cs}
\icmlauthor{Ziv Goldfeld}{ece}
\icmlauthor{Kengo Kato}{stats}
\end{icmlauthorlist}

\icmlaffiliation{cs}{Department of Computer Science, Cornell University, Ithaca, NY}
\icmlaffiliation{ece}{School of Electrical and Computer Engineering, Cornell University, Ithaca, NY}
\icmlaffiliation{stats}{Department of Statistics and Data Science, Cornell University, Ithaca, NY}

\icmlcorrespondingauthor{Sloan Nietert}{nietert@cs.cornell.edu}

\icmlkeywords{optimal transport, empirical approximation, Wasserstein distance}

\vskip 0.3in
]

\printAffiliationsAndNotice{}  %

\begin{abstract}
Discrepancy measures between probability distributions, often termed statistical distances, are ubiquitous in probability theory, statistics and machine learning. To combat the curse of dimensionality when estimating these distances from data, recent work has proposed smoothing out local irregularities in the measured distributions via convolution with a Gaussian kernel. Motivated by the scalability of this framework to high dimensions, we investigate the structural and statistical behavior of the Gaussian-smoothed $p$-Wasserstein distance $\GWp$, for arbitrary $p\geq 1$.
After establishing basic metric and topological properties of $\GWp$, we explore the asymptotic statistical behavior of $\GWp(\hat{\mu}_n,\mu)$, where $\hat{\mu}_n$ is the empirical distribution of $n$ independent observations from $\mu$.
We prove that $\GWp$ enjoys a parametric empirical convergence rate of $n^{-1/2}$, which contrasts the $n^{-1/d}$ rate for unsmoothed $\Wp$ when $d \geq 3$. Our proof relies on controlling $\GWp$ by a $p$th-order smooth Sobolev distance $\Gds$ and deriving the limit distribution of $\sqrt{n}\,\Gds(\hat{\mu}_n,\mu)$, for all dimensions $d$. As applications, we provide asymptotic guarantees for two-sample testing and minimum distance estimation using $\GWp$, with experiments for $p=2$ using a maximum mean discrepancy formulation~of~$\mathsf{d}_2^{(\sigma)}$. 
\end{abstract}

\section{Introduction}\label{sec:intro}

The Wasserstein distance $\Wp$ is a discrepancy measure between probability distributions rooted in the theory of optimal transport \cite {villani2003,villani2008optimal}.
It has seen a surge of applications in statistics and ML, ranging from generative modeling \cite{arjovsky_wgan_2017, gulrajani2017improved, tolstikhin2018wasserstein} and image recognition \cite{rubner2000earth,sandler2011nonnegative,li2013novel} to domain adaptation \cite{courty2014domain,courty2016optimal} and robust optimization \cite{mohajerin_robust_2018,blanchet2018optimal, gao2016distributionally}.
The widespread use of this statistical distance is driven by an array of desirable properties, including its metric structure, a convenient dual form, and its robustness to support mismatch.

In applications, the Wasserstein distance is often estimated from samples.
However, the error of these empirical estimates suffers from an exponential dependence on dimension that presents an obstacle to sample-efficient bounds for inference and learning. More specifically, the rate at which $\Wp(\hat{\mu}_n,\mu)$ converges to 0, where $\hat{\mu}_n$ is an empirical measure based on $n$ independent %
samples from $\mu$, scales as $n^{-1/d}$ under mild moment conditions, for $d \geq 3$ \cite{dereich2013,boissard2014mean,fournier2015,panaretos2019statistical,weed2019rate, lei2020convergence}. This rate deteriorates poorly with dimension and seems at odds with the scalability of empirical $\Wp$ observed in modern machine learning practice.

\subsection{Smooth Wasserstein Distances}

Gaussian smoothing was recently introduced as a means to alleviate the  curse of dimensionality of empirical $\Wp$, while preserving the virtuous structural properties of the classic framework \cite{goldfeld2019,Goldfeld2020GOT,goldfeld2020asymptotic}.
Specifically the $\sigma$-smooth $p$-Wasserstein distance is $\GWp(\mu,\nu) := \Wp(\mu*\Gauss,\nu*\Gauss)$, where $\Gauss = \cN(0,\sigma^2 \I)$ is the isotropic Gaussian measure of parameter $\sigma$. \citet{Goldfeld2020GOT} showed that $\gwass$ inherits the metric and topological structure of~$\mathsf{W}_1$ and approximates it within an $O(\sigma\sqrt{d})$ gap. At the same time, empirical convergence rates for $\GWp$ are much faster. %
As shown in \cite{goldfeld2019}, $\E\big[\gwass(\hat{\mu}_n,\mu)\big] = O(n^{-1/2})$ when $\mu$ is sub-Gaussian in any dimension, i.e., it exhibits a parametric convergence rate. This fast rate was also established for $\mathsf{W}_2^{(\sigma)}$ but only when the sub-Gaussian constant is smaller than $\sigma/2$. These results significantly depart from the $n^{-1/d}$ rate in the unsmoothed case.

Follow-up work \cite{goldfeld2020asymptotic} developed a limit distribution theory for $\gwass$, showing that $\sqrt{n}\,\gwass(\hat{\mu}_n,\mu)$ converges in distribution to the supremum of a tight Gaussian process, for all dimensions $d$ and under a milder moment condition.
Such limit distribution results are known for unsmoothed $\mathsf{W}_p$ only when $p \in \{ 1,2 \}$ and $d=1$ \citep{delbarrio1999,delbarrio2005} or $\mu$ is supported on a finite or countable set \citep{sommerfeld2018,tameling2019}, but are a wide open question otherwise. 
Other works investigated the behavior of $\GWp$
as $\sigma \to \infty$ \cite{chen2020}, and established related results for other statistical distances, including total variation (TV), Kullback-Leibler (KL) divergence, and $\chi^2$-divergence \cite{goldfeld2019,chen2020}.

\subsection{Contributions}

We focus on the smooth $p$-Wasserstein distance $\GWp$ for $p > 1$ and arbitrary dimension $d$. We first explore basic structural properties of $\GWp$, proving that many of the beneficial attributes of $\Wp$ carry over to the smooth setting. We show that $\GWp$ is a metric and induces the same topology as $\Wp$.
Then, we prove that $\GWp$ is stable under small perturbations of the smoothing parameter $\sigma$, implying, in particular, that $\GWp \to \Wp$ as $\sigma \to 0$. 
We then extend the stability of optimal transport distances to that of transport plans, establishing weak convergence of the optimal couplings for $\GWp$ to those of $\Wp$ as $\sigma$ shrinks.

Moving on to a statistical analysis, we explore empirical convergence for $\GWp$. Elementary techniques imply that $\E\big[\GWp(\hat{\mu}_n,\mu)\big] = O(n^{-1/(2p)})$ under a mild moment condition. While this rate is independent of $d$, it is suboptimal in $p$, with the expected answer being $n^{-1/2}$ as previously established for $p=1,2$ \citep{goldfeld2020asymptotic,goldfeld2019}.
To get the correct rate, we establish a comparison between $\GWp$ and a smooth $p$th-order Sobolev integral probability metric (IPM), $\Gds$; the latter lends itself well to tools from empirical process theory.
Under a sub-Gaussian assumption, we prove that the function class defining $\Gds$ is $\mu$-Donsker, giving a limit distribution for $\sqrt{n}\, \Gds(\hat{\mu}_n,\mu)$ that implies the  $n^{-1/2}$ rate for $\GWp$.
We conclude with a concentration inequality for $\GWp(\hat{\mu}_n,\mu)$.

We next turn to computational aspects,
first showing that $\mathsf{d}_2^{(\sigma)}$ is efficiently computable as a maximum mean discrepancy (MMD) and characterizing its reproducing kernel. 
Next, we consider applications to two-sample testing and generative modeling using $\GWp$.
We construct two-sample tests based on the smooth $p$-Wasserstein test statistic %
that achieve asymptotic consistency and correct asymptotic level.
For generative modeling, we examine minimum distance estimation with $\GWp$
and establish measurability, consistency, and parametric convergence rates, along with finite-sample generalization guarantees in arbitrary dimension.
Many of these directions (beyond measurability and consistency) are intractable with standard $\Wp$ unless $d = 1$.
We conclude with numerical results that support our theory.

\subsection{Related Discrepancy Measures}
The sliced Wasserstein distance \cite{rabin2011} takes an average (or maximum \cite{deshpande2019}) of one-dimensional Wasserstein distances over random projections of the $d$-dimensional distributions.
Like the smooth framework considered herein, the sliced distance also exhibits an $n^{-1/2}$ empirical convergence rate and has characterized limit distributions in some cases \cite{nadjahi2019,nadjahi2020}.
In \cite{nadjahi2019}, sliced $\mathsf{W}_1$ was shown to be a metric that induces a topology at least as fine as that of weak convergence, akin to $\gwass$. They further examined generative modeling via minimum sliced Wasserstein estimation, establishing measurability, consistency, and some limit theorems.
However, while $\GWp$ converges to $\Wp$ as $\sigma \to 0$, there is no approximation parameter for these sliced distances, and comparisons to the standard distances typically require compact support and feature dimension-dependent multiplicative constants (see, e.g., \cite{bonnotte2013}).

Another relevant framework is entropic optimal transport (EOT), which  admits efficient algorithms \cite{cuturi2013,altschuler2017} and some desirable statistical properties \cite{genevay2016,rigollet2018}. In particular, two-sample EOT has empirical convergence rate $n^{-1/2}$ for smooth costs and compactly supported distributions \cite{genevay2019}.
For the quadratic cost, \cite{mena19} extended this rate to sub-Gaussian distributions and derived a central limit theorem (CLT) for empirical EOT, mirroring a result for $\mathsf{W}_2$ established in \cite{barrio2019}.
The \citet{mena19} CLT is notably different from ours: it uses as a centering constant the expected empirical distance between $\hat{\mu}_m$ and $\hat{\nu}_n$ as opposed to the population distance between $\mu$ and $\nu$ (which corresponds to our centering about 0 in the one-sample case).
Finally, while EOT can be computed efficiently, it is no longer a metric, even if the underlying cost is \cite{feydy2019,bigot2019}.

Sobolev IPMs have proven independently useful for generative modeling, often referred to as `dual Sobolev norms'. For example, alternative Sobolev IPMs are the basis for multiple generative adversarial network (GAN) frameworks \cite{youssef2018, xu2020} and are featured in \cite{si2020}, which examines Wasserstein projections of empirical measures onto a chosen hypothesis class.

\section{Preliminaries}
\label{sec:background}

\subsection{Notation}
Let $| \cdot |$ and $\langle \cdot, \cdot \rangle$ denote the Euclidean norm and inner product. For a (signed) measure $\mu$ and a measurable function $f$ on $\R^d$, we write $\mu(f) = \int f \, \dd \mu$.
For a non-empty set $\cT$, let $\ell^\infty(\cT)$ be the space of all bounded functions $f:\cT \to \R$, equipped with the sup-norm $\|f\|_{\infty,\cT} = \sup_{t \in \cT}|f(t)|$.
The space of compactly supported, infinitely differentiable real functions on $\R^d$ is~$C_0^\infty$.
For any $p \in [1,\infty)$ and any Borel measure $\gamma$ on $\R^d$, we denote by  $L^p(\gamma;\R^k)$ the space of measurable maps $f: \R^d \to \R^k$ such that $\|f\|_{L^p(\gamma;\R^k)} = (\int_{\R^d}|f|^p \dd \gamma)^{1/p} < \infty$. The space $(L^p(\gamma;\R^k), \| \cdot \|_{L^p(\gamma;\R^k)})$ is a Banach space, and we also write $L^p(\gamma) = L^p(\gamma;\R^1)$.

The class of Borel probability measures on $\R^d$ is $\cP$ and the subset of measures $\mu \in \cP$ with finite $p$-th moment $\int |x|^{p} \, \dd \mu(x)$ is $\cP_p$. The convolution of measures $\mu,\nu \in \cP$ is defined by $(\mu*\nu)(A) := \int\int \mathds{1}_A(x + y) \, \dd \mu(x) \, \dd \nu(y)$, where $\mathds{1}_A$ is the indicator of $A$. The convolution of measurable functions $f,g$ on $\R^d$ is $(f*g)(x) := \int f(x-y)g(y) \, \dd y$. Recall that $\Gauss = \cN(0,\sigma^2 \I)$ and use $\varphi_\sigma(x) = (2\pi\sigma^2)^{-d/2}e^{-|x|^2/(2\sigma^2)}$, $x \in \R^d$, for the Gaussian density. Write $\mu \otimes \nu$ for the product measure of $\mu,\nu \in \cP$. Let $\stackrel{w}{\to}$ and $\stackrel{d}{\to}$ denote weak convergence of probability measures and convergence in distribution of random variables.

\subsection{Background}

We next provide some background on the statistical distances used in this paper.

\paragraph{(Smooth) Wasserstein Distance.} For $p \geq 1$, the $p$-Wasserstein distance between $\mu, \nu \in \cP_p$ is defined by
\[
\Wp(\mu,\nu) := \left(\inf_{\pi \in \Pi(\mu,\nu)} \int_{\R^d}|x-y|^p\, \dd \pi(x,y)\right)^{1/p}
\]
where $\Pi(\mu,\nu)$ is the set of couplings of $\mu$ and $\nu$.
See \cite{villani2003,villani2008optimal,santambrogio2015} for additional background.
The $\sigma$-smooth $p$-Wasserstein distance between probability measures $\mu, \nu \in \cP_p$ is defined by 
\[\GWp(\mu,\nu) := \Wp(\mu*\Gauss, \nu*\Gauss).\]

\paragraph{Integral Probability Metrics.}
Let $\cF$ be a class of measurable real functions on $\R^d$. 
The IPM with respect to (w.r.t.) $\cF$ between probability measures $\mu,\nu \in \cP$ is defined by
\[
\|\mu - \nu\|_{\infty,\cF} = \sup_{f \in \cF} \mu(f) - \nu(f).
\]
We subsequently control $\GWp$ via an IPM whose functions have bounded Sobolev norm.

\paragraph{Smooth Sobolev IPM.}

Let $\gamma$ be a Borel measure on $\R^d$ and fix $p \geq 1$.
For a differentiable function $f:\R^{d}\to\R$, let
\[
\| f \|_{\dot{H}^{1,p}(\gamma)}:= \| \nabla f \|_{L^{p}(\gamma;\R^d)} = \left (\int_{\R^{d}} |\nabla f |^p  \dd\gamma \right)^{1/p}
\]
be its Sobolev seminorm. 
We define the homogeneous Sobolev space $\dot{H}^{1,p}(\gamma)$ as the completion of $\dot{C}_{0}^\infty = \{ f+a : a \in \R, f \in C_0^\infty \}$ w.r.t. $\| \cdot \|_{\dot{H}^{1,p}(\gamma)}$. 
The dual Sobolev norm of a signed measure $\ell$ on $\R^d$ with zero total mass is
\[
\|\ell\|_{\dot{H}^{-1,p}(\gamma)} := \sup \{ \ell(f) : f \in C_0^\infty, \|f\|_{\dot{H}^{1,q}(\gamma)} \leq 1 \},
\]
where $q$ is the conjugate index of $p$, i.e., $1/p+1/q = 1$. 
We define the $p$th-order smooth Sobolev IPM by
\[
\Gds(\mu,\nu) := \|(\mu - \nu)*\Gauss\|_{\dot{H}^{-1,p}(\Gauss)}
\]
for measures $\mu,\nu \in \cP$.
Observe that $\Gds$ is an IPM w.r.t. the class $\cF * \varphi_\sigma = \{ f * \varphi_\sigma : f \in \cF \}$ with $\cF = \{ f \in C_0^\infty : \|f\|_{\dot{H}^{1,q}(\Gauss)} \leq 1\}$.

\section{Structure of Smooth Wasserstein Distance and Comparison with Smooth Sobolev IPM}
\label{sec:smooth-Wp}

We now examine basic properties of smooth Wasserstein distances, including a useful connection to the smooth Sobolev IPM. 
The case of $\mathsf{W}_1^{(\sigma)}$ has been well-studied in \cite{Goldfeld2020GOT,goldfeld2020asymptotic}. Herein we present results that hold for arbitrary $p \geq 1$ and $\sigma \geq 0$ unless stated otherwise, with proofs left for the appendix. Extending beyond $p=1$ requires new techniques, most prominently a comparison result between $\GWp$ and $\Gds$.

\subsection{Structural Properties}

We first consider the topology induced by $\GWp$. Since convolution acts as a contraction, we have $\Wp^{(\sigma)} \leq \Wp$. In fact, the two distances induce the same topology on~$\cP_{p}$, which coincides with that of weak convergence in addition to convergence of $p$th moments.

\begin{proposition}[Metric and topological structure of $\GWp$]
\label{prop:smooth-wasserstein-metric}
$\Wp^{(\sigma)}$ is a metric on $\cP_{p}$ inducing the same topology as $\Wp$.
\end{proposition}

The proof uses existence of optimal couplings and uniform integrability arguments.
Next, we examine the behavior of $\GWp$ as a function of the smoothing parameter $\sigma$.
We start from the following stability lemma, guaranteeing that small changes in $\sigma$ result only in slight perturbations of $\GWp$.

\begin{lemma}[Stability of $\GWp$]
\label{lem:stability}
For $\mu,\nu \in \cP_p$ and $0 \le \sigma_1 \leq \sigma_2 < \infty$, we have
\begin{align*}
\big|\Wp^{(\sigma_2)}(\mu,\nu) -\Wp^{(\sigma_1)}(\mu,\nu)\mspace{-1mu}\big| \leq 2\sqrt{(\sigma_2^2 \mspace{-1mu} - \mspace{-1mu} \sigma_1^2)(d \mspace{-1mu}+\mspace{-1mu} 2p\mspace{-1mu} + \mspace{-1mu}2)}.
\end{align*}
\end{lemma}

This result generalizes Lemma 1 of \cite{Goldfeld2020GOT}, which covers $p=1$, and establishes uniform continuity of $\GWp$ in $\sigma$. Its proof takes a different approach, using Minkowski's inequality instead of the Kantorovich-Rubinstein duality. An immediate consequence of \cref{lem:stability} is given next, mirroring Theorem 3 of \cite{Goldfeld2020GOT}. %

\begin{corollary}
[$\GWp$ dependence on $\sigma$]
\label{prop:smooth-Wp-dependence-on-sigma}
For~$\mu,\nu \in \cP_p$, the following hold:
\vspace{-3.5mm}
\begin{enumerate}[wide, labelindent=0pt, label=(\roman*)]
\item $\GWp(\mu,\nu)$ is continuous and monotonically non-increasing in $\sigma \in [0,+\infty)$;
\vspace{-1.5mm}
\item $\lim\limits_{\sigma \to 0} \GWp(\mu,\nu) = \Wp(\mu,\nu)$;
\vspace{-1.5mm}
\item $\lim\limits_{\sigma \to \infty} \GWp(\mu,\nu) = \big|\E[X] - \E[Y]\big|\:$, for $X \sim \mu$ and $Y \sim \nu$ sub-Gaussian.
\end{enumerate}
\end{corollary}

\begin{remark}[Infinite smoothing]
A detailed study of $\GWp$ in the infinite smoothing regime (i.e., when $\sigma \to \infty$) is conducted in \cite{chen2020}. Therein, the authors prove Item 3 above and examine the convergence of $\GWp(\mu,\nu)$ to 0 when $\E[X]=\E[Y]$. For that case, they show that if $\mu$ and $\nu$ have matching moment tensors up to order $n$ (but not $n+1$), then $\mathsf{W}_2^{(\sigma)}(\mu,\nu) \asymp \sigma^{-n}$ as $\sigma \to \infty$.
\end{remark}

\cref{prop:smooth-Wp-dependence-on-sigma} guarantees the convergence of transport costs as $\sigma \to 0$. It is natural to ask whether optimal transport plans (i.e., couplings) that achieve these costs converge as well. We answer this question to the affirmative.%
 
\begin{proposition}[Convergence of transport plans]\label{prop:convergence-of-plans}
Fix $\mu,\nu \in \cP_p$ and let $(\sigma_k)_{k \in \N}$ be a sequence with $\sigma_k \searrow \sigma \geq 0$. For each $k \in \N$, let $\pi_k \in \Pi(\mu * \cN_{\sigma_k}, \nu*\cN_{\sigma_k})$ be an optimal coupling for $\mathsf{W}_p^{(\sigma_k)}(\mu,\nu)$. Then there exists $\pi \in \Pi(\mu*\Gauss, \nu*\Gauss)$ such that $\pi_k \stackrel{w}{\to} \pi$ as $k \to \infty$ (up to subsequences if $\sigma = 0$ or $p=1$), and $\pi$ is optimal for $\GWp(\mu,\nu)$.
\end{proposition}

The proof  observes that the arguments for Theorem 4 of \cite{Goldfeld2020GOT} extend from $p=1$ to the general case with minor changes.

So far we have studied metric, topological, and limiting properties of $\GWp$. In \cref{sec:empirical} we explore its statistical behavior, when distributions are estimated from samples.
To that end, we now establish a relation between $\GWp$ and the smooth Sobolev IPM $\Gds$. 
This result is later used to study the empirical convergence under $\GWp$ using tools from empirical process theory (applied to the $\Gds$ upper~bound).

\begin{theorem}[Comparison between $\GWp$ and $\Gds$]
\label{thm:smooth-Wp-Sobolev-comparison}
Fix ${p > 1}$ and let $q$ be the conjugate index of $p$.
Then, for $X \sim \mu \in \cP_{p}$ with mean 0 and $\nu \in \cP$, we have
\begin{equation}
\GWp(\mu,\nu) \le p\, e^{\E[|X|^2]/(2q\sigma^2)}\,\Gds\big(\mu,\nu\big).
\label{eq:duality}
\end{equation}
\end{theorem}

The proof builds upon related inequalities established for standard $\Wp$ \cite{dolbeault2009, peyre2018, ledoux2019}, exploiting the metric structure of the Wasserstein space and the Benamou-Brenier dynamic formulation of optimal transport  \cite{benamou2000}. Namely, we note that $\Wp(\mu_0,\mu_1)$ is upper bounded by the length of any continuous path from $\mu_0$ to $\mu_1$ in $(\cP_p,\Wp)$ and examine the path $t \mapsto t\mu_1 + (1-t)\mu_0$ which interpolates linearly between the two densities.
The theorem follows upon applying the resulting bound to $\mu*\Gauss$ and $\nu*\Gauss$.
We also give a lower bound for $\GWp(\mu,\nu)$ using $\|(\mu - \nu)*\cN_{\sigma}\big\|_{\dot{H}^{-1,p}(\cN_{\sqrt{2}\sigma})}$, though the constant factor restricts its usefulness.

\begin{remark}
When $p=1$, one can show that $\mathsf{W}_1$ and the dual Sobolev norm $\| \cdot \|_{\dot{H}^{-1,1}(\gamma)}$ coincide \cite{dolbeault2009}. In particular, this implies that $\mathsf{W}_1^{(\sigma)}(\mu,\nu) = \mathsf{d}_1^{(\sigma)}(\mu,\nu)$. 
For larger $p$, the gap between $\GWp(\mu,\nu)$ and the upper bound given by \cref{thm:smooth-Wp-Sobolev-comparison} can grow quite large, so we view the comparison as a useful theoretical tool rather than a device for practical approximation guarantees.
\end{remark}

Finally, we establish some basic properties of $\Gds$.
\begin{proposition}[$\Gds$ dependence on $\sigma$]
\label{prop:Gds-sigma-dependence}
For $\mu,\nu \in \cP$, the following hold:
\begin{enumerate}[wide, labelindent=0pt, label=(\roman*)]
    \item $\lim_{\sigma \to 0} \Gds(\mu,\nu) = \infty$ for $\mu \neq \nu$;
    \item $\lim_{\sigma \to \infty} \mathsf{d}_2^{(\sigma)}(\mu,\nu) = |\E[X] - \E[Y]|$, for $X \sim \mu$ and $Y \sim \nu$ sub-Gaussian.
\end{enumerate}
\end{proposition}
We focus on $p=2$ for (ii) due to a convenient MMD formulation for $\mathsf{d}_2^{(\sigma)}$ established in \cref{sec:computation}.

\section{Empirical Approximations}\label{sec:empirical}

Fix  $p > 1$, $\sigma>0$, and let $\mu \in \cP_p$ with $X \sim \mu$.
Given independently and identically distributed (i.i.d.) samples $X_1, \dots, X_n \sim \mu$ with empirical distribution $\hat{\mu}_n := n^{-1} \sum_{i=1}^n \delta_{X_i}$, we study the convergence rate of $\E\big[\Wp^{(\sigma)}(\hat{\mu}_n,\mu)\big]$ to zero. To start, we observe that elementary techniques imply $\E\big[\Wp^{(\sigma)}(\hat{\mu}_n,\mu)\big] = O(n^{-1/(2p)})$ under mild conditions on $\mu$. Although the rate $n^{-1/(2p)}$ is  dimension-free, its dependence on $p$ is sub-optimal.

\begin{theorem}[Slow rate]\label{thm:slow-rate}
If $X\sim\mu$ satisfies
\begin{equation}
\int_{0}^{\infty} r^{d+p-1}\sqrt{\Prob(|X| > r)} \dd r < \infty,
\label{eq:poly-moment}
\end{equation}
then $\E\big[\GWp(\hat{\mu}_n,\mu)\big] = O(n^{-1/(2p)})$.
Condition (\ref{eq:poly-moment}) holds if $\mu$ has finite $(2d+2p+\epsilon)$-th moment for some $\epsilon > 0$. 
\end{theorem}
The proof 
follows by coupling $\mu$ and $\hat{\mu}_n$ via the maximal coupling. This bounds $\big(\GWp(\hat{\mu}_n,\mu)\big)^p$ from above by a weighted TV distance, which converges as $n^{-1/2}$, provided that the above moment condition holds.
This proof technique was previously applied in \cite{goldfeld2019} to achieve the same rate when $p=1$.

We next turn to show that the $n^{-1/2}$ rate is attainable for $\GWp(\hat{\mu}_n,\mu)$ itself (rather than for its $p$th power).
To this end, we first establish a limit distribution result for the empirical smooth Sobolev IPM $\Gds(\hat{\mu}_n,\mu)$. This, in turn, yields the desired rate for $\E\big[\GWp(\hat{\mu}_n,\mu)\big]$ via the comparison from \cref{thm:smooth-Wp-Sobolev-comparison}.
Recall the function class $\cF = \{  f \in C_0^\infty : \|f\|_{\dot{H}^{1,q}(\Gauss)} \leq 1 \}$.

\begin{theorem}[Limit distribution for empirical $\Gds$]
\label{thm:limit-distribution}
Suppose there exists $\theta>p-1$ for which $X\sim\mu$ satisfies
\begin{equation}
\label{eq:moment-condition}
\int_{0}^{\infty} e^{\frac{\theta r^2}{2\sigma^2}} \sqrt{\Prob (|X| > r)} \dd r < \infty.
\end{equation}
Then $\sqrt{n}\Gds(\hat{\mu}_n,\mspace{-2mu}\mu)\mspace{-3mu}\stackrel{d}{\to}\mspace{-3mu}\|G \|_{\infty,\cF}$ as $n\mspace{-2mu}\to\mspace{-2mu}\infty$,
where \ $G = (G(f))_{f \in \cF}$ is a tight Gaussian process in $\ell^{\infty}(\cF)$ with mean zero and covariance function $\Cov (G(f),G(g)) = \Cov (f*\varphi_{\sigma}(X),g*\varphi_{\sigma}(X))$. %
\end{theorem}

\begin{corollary}[Fast rate]
\label{cor:fast-rate}
Under the conditions of \cref{thm:limit-distribution}, 
we have $\lim_{n \to \infty}\mspace{-3mu} \sqrt{n}\E\big[\Gds(\hat{\mu}_n,\mu)\big]\mspace{-3mu}=\mspace{-3mu}\E\mspace{-3mu}\big[\| G \|_{\infty,\cF}\big]\mspace{-3mu}<\mspace{-3mu}\infty$. Consequently, $\E\big[\GWp(\hat{\mu}_n,\mu)\big] = O(n^{-1/2})$.
\end{corollary}

The proof of \cref{thm:limit-distribution} shows that
the smoothed function class $\cF * \varphi_\sigma = \{ f* \varphi_{\sigma} : f \in \cF \}$ is $\mu$-Donsker. 
Specifically, we prove that functions in $\cF * \varphi_\sigma$ are smooth with derivatives uniformly bounded on domains within a fixed radius of the origin. Using these bounds, we apply techniques from empirical process theory to establish the Donsker property. 
Importantly, the preceding argument hinges on the convolution with the smooth Gaussian density and does \emph{not} hold for the unsmoothed function class.
No mean zero requirement appears because we can center $\hat{\mu}_n$ and $\mu$ by the mean of $\mu$.

Condition \eqref{eq:moment-condition} requires that $\Prob(|X| > r)\to 0$ faster than $e^{-Cr^2}$ as $r \to \infty$ for some $C > (p-1)/\sigma^2$, which in turn requires $|X|$ to be sub-Gaussian.
The requirement is trivially satisfied if $\mu$ is compactly supported.
We can also relate Condition (\ref{eq:moment-condition}) to a more standard notion of sub-Gaussianity for random vectors.

\begin{definition}[Sub-Gaussian distribution]
Let $Y\sim\nu\in\cP$ with $\E[|Y|]<\infty$. We say that $\nu$ or $Y$ is \emph{$\beta$-sub-Gaussian} for $\beta \ge 0$ if $\E[\exp ( \langle \alpha, Y - \E[Y]\rangle)] \le \exp (\beta|\alpha|^2/2)$ for all $\alpha \in \R^{d}$. 
\end{definition}

\begin{proposition}[Sub-Gaussianity implies \eqref{eq:moment-condition}]
\label{lem:subGaussian}
If $\mu$ is $\beta$-sub-Gaussian with $\beta < \sigma/\sqrt{2(p-1)}$, then \eqref{eq:moment-condition} holds. 
\end{proposition}

Next, we consider the concentration of $\GWp(\hat{\mu}_n,\mu)$.

\begin{proposition}[Concentration inequality]
\label{prop:concentration1}
If $\mu$ has compact support, then, for all $t > 0$ and $n \in \N$, we have
\[
\Prob\left(\GWp(\hat{\mu}_n, \mu) \geq Cn^{-1/2} + t\right) \leq \exp(-cnt^2)
\]
with constants $C,c$ independent of $n$ and $t$.
\end{proposition}

For unbounded domains, concentration results mirroring those of Corollary 3 in \cite{goldfeld2020asymptotic} can be established in the same way under a stronger sub-Gaussianity assumption; we omit the details for brevity.

\begin{remark}[Constants]
While the rates provided in this section are dimension-free, the constants necessarily exhibit an exponential dependence on dimension. Indeed, minimax results for estimation of standard $\Wp$ due to \citet{singh2018}, combined with \cref{lem:stability},
imply that achieving dimension-free rates with constants scaling only polynomially in dimension is impossible in general.
We provide further details in the proofs of \cref{thm:slow-rate} and \cref{thm:limit-distribution}.
\end{remark}

\section{Smooth Sobolev IPM Efficient Computation}%
\label{sec:computation}
We next consider computation of $\Gds$, which takes a convenient form when $p=2$. 
Specifically, the Hilbertian structure of $\dot{H}^{1,2}(\Gauss)$ enables to streamline calculations significantly for all dimensions $d$.
In the following, fix $\sigma > 0$, $\mu,\nu \in \cP$, $X,X' \sim \mu \otimes \mu$, and $Y,Y' \sim \nu \otimes \nu$.

Consider the function space $\dot{H}_{0}^{1,2}(\Gauss) := \{f \in \dot{H}^{1,2}(\Gauss)$ $: \Gauss(f) = 0 \}$ with norm $\| \cdot \|_{\dot{H}^{1,2}(\Gauss)}$ (this norm is proper because of the constraint $\Gauss(f)=0$). 
This space becomes a Hilbert space when equipped with inner product $\langle f,g \rangle_{\dot{H}^{1,2}(\Gauss)} = \int_{\R^{d}} \langle \nabla f, \nabla g \rangle \, \dd \Gauss$.
Likewise, the space $\dot{H}_0^{1,2}(\Gauss) * \varphi_\sigma = \{ f*\varphi_\sigma : f \in \dot{H}_0^{1,2}(\Gauss) \}$ is a Hilbert space with inner product $\langle f*\varphi_\sigma,g*\varphi_\sigma \rangle_{\dot{H}^{1,2}(\Gauss) * \varphi_\sigma} = \langle f,g \rangle_{\dot{H}^{1,2}(\Gauss)}$ (see Appendix \ref{prfs:computation} for a proof that this is well-defined). In fact, we can say a bit more, first recalling some definitions.

\paragraph{Reproducing Kernel Hilbert Space (RKHS).}
Let $\cH$ be a Hilbert space of real-valued functions on $\R^d$.
We say that $\cH$ is an RKHS if  there is a positive semidefinite function $k:\R^d \times \R^d \to \R$, called a reproducing kernel,
such that $k(\cdot,x) \in \cH$ and $h(x) = \langle h, k(\cdot,x) \rangle$ for all $x \in \R^d$ and $h \in \cH$.
See \cite{steinwart2008, aronszajn_rkhs_1950} for comprehensive background on RKHSs.

\paragraph{Maximum Mean Discrepancy.}
Let $\cH$ be an RKHS with kernel $k$. The IPM corresponding to its unit ball, termed MMD, is given by
\begin{equation*}
    \mathsf{MMD}_\cH(\mu,\nu) := \sup_{f \in \cH :\, \|f\|_\cH \leq 1} \mu(f) - \nu(f).
\end{equation*}

\begin{proposition}[\citet{borgwardt2006}]
If $\E\bigl[\mspace{-4mu}\sqrt{k(X,X)}\bigr] \mspace{-1mu},$ $\E\bigl[\mspace{-4mu}\sqrt{k(Y,Y)}\bigr]<\infty$, then
\begin{equation}\label{eq:mmd}
    \mathsf{MMD}_\cH(\mu,\mspace{-2mu}\nu)^2\mspace{-3mu}=\mspace{-3mu}\E[k(X,\mspace{-2mu}X')]- 2\mspace{-2mu}\E[k(X,\mspace{-2mu}Y)]+ \E[k(Y,\mspace{-2mu}Y')]\mspace{-2mu}.
\end{equation}
\end{proposition}

We prove that $\dot{H}_0^{1,2}(\Gauss) * \varphi_\sigma$ is an RKHS whose kernel is expressed in terms of the entire exponential integral \cite{oldam_atlas_2009} $\Ein(z) := \int_{0}^{z}\left(1-e^{-t}\right) \frac{d t}{t}=\sum_{k=1}^{\infty} \frac{(-1)^{k+1} z^{k}}{k \cdot k !}$, giving an MMD form for $\mathsf{d}_2^{(\sigma)}$.

\begin{theorem}[$\mathsf{d}_2^{(\sigma)}$ as an MMD]
\label{thm:dual-Sobolev-RKHS}
The space $\dot{H}_0^{1,2}(\Gauss) * \varphi_\sigma$ is an RKHS with reproducing kernel
$\kappa^{(\sigma)}(x,y):= - \sigma^2 \Ein\left(-\langle x,y \rangle/\sigma^2\right)$.
Thus, if $\E\big[\sqrt{\kappa^{(\sigma)}(X,X)}\big], \E\big[\sqrt{\kappa^{(\sigma)}(Y,Y)}\big] < \infty$, then
\begin{equation}
\begin{split}
\label{eq:MMD}
\mathsf{d}_2^{(\sigma)}(\mu,\nu)^2 = \E\big[\kappa^{(\sigma)}(X,X')\big] &+ \E\big[\kappa^{(\sigma)}(Y,Y')\big]\\ &- 2\E\big[\kappa^{(\sigma)}(X,Y)\big].
\end{split}
\end{equation}
\end{theorem}

The proof begins with a reduction to $\sigma = 1$ and observes that properly normalized multivariate Hermite polynomials form an orthonormal basis for $\dot{H}_0^{1,2}(\cN_1)$.
Convolving these polynomials with $\varphi_1$, we obtain an orthonormal basis for $\dot{H}_0^{1,2}(\cN_1)*\varphi_1$ comprising scaled monomials, which can then be used to calculate the kernel.

The MMD formulation \eqref{eq:MMD} gives a convenient way to compute $\mathsf{d}_2^{(\sigma)}$ in practice. Suppose that we generate i.i.d. samples $X_1,\dots,X_m \sim \mu$ and $Y_1,\dots,Y_n \sim \nu$ with empirical distributions $\hat{\mu}_m = m^{-1} \sum_{i=1}^{m} \delta_{X_i}$ and $\hat{\nu}_n = n^{-1} \sum_{j=1}^{n} \delta_{Y_j}$. Then, we can  compute
\begin{multline*}
\mathsf{d}_2^{(\sigma)}(\hat{\mu}_m,\hat{\nu}_n)^2 = \frac{1}{m^2} \sum_{i=1}^m \sum_{j=1}^m \kappa^{(\sigma)}(x_i,x_j) + \\ \frac{1}{n^2} \sum_{i=1}^n \sum_{j=1}^n \kappa^{(\sigma)}(y_i,y_j) - \frac{2}{m n} \sum_{i=1}^m \sum_{j=1}^n \kappa^{(\sigma)}(x_i,y_j).
\end{multline*}
Provided that $\mu$ and $\nu$ are compactly supported or $\beta$-sub-Gaussian with $\beta < \sigma/\sqrt{2}$, \cref{cor:fast-rate} and the triangle inequality imply
\[
\E\left[\,\left|\mathsf{d}_2^{(\sigma)}(\hat{\mu}_m,\hat{\nu}_n) - \mathsf{d}_2^{(\sigma)}(\mu,\nu) \right|\,\right]= O\left( \min\{m,n\}^{-1/2} \right).
\]
Hence, we can approximate $\mathsf{d}_2^{(\sigma)}$ up to expected error $\eps$ with $O(\eps^{-2})$ samples from the measured distributions and $O(\eps^{-4})$ evaluations of $\kappa^{(\sigma)}$ for any dimension $d$.

\section{Statistical Applications}
\label{sec:applications}

With the empirical approximation and computational results in hand, we now present applications to two-sample testing and minimum distance estimation. These results highlight the benefits of smoothing, as several of the subsequent claims are unavailable for standard $\Wp$ due to the lack of parametric rates and limit distributions. 

\subsection{Two-Sample Testing}
\label{subsec:two-sample-testing}

We start from two-sample testing with $\GWp$ and $\Gds$, where $p > 1$ and $\sigma > 0$ are fixed throughout. Let $\mu,\nu \in \cP_p$ and take $X_1, \dots, X_m \sim \mu$ and $Y_1, \dots, Y_n \sim \nu$ to be mutually independent samples. %
The goal of nonparametric two-sample testing is to detect, based on the samples, whether the null hypothesis $H_0: \mu = \nu$ holds, without imposing parametric assumptions on the distributions.

A standard class of tests rejects $H_0$ if $D_{m,n} > c_{m,n}$, where $D_{m,n} = D_{m,n}(X_1, \dots, X_m, Y_1, \dots, Y_n)$ is a scalar test statistic and $c_{m,n}$ is a critical value chosen according to the desired level $\alpha \in (0,1)$. 
Precisely, we say that such a sequence of tests has asymptotic \emph{level} $\alpha$ if $\limsup_{m,n\to\infty} \Prob (D_{m,n} > c_{m,n}) \leq \alpha$ whenever $\mu = \nu$.
We say that these tests are asymptotically \emph{consistent} if $\lim_{m,n\to\infty} \Prob (D_{m,n} > c_{m,n}) = 1$ whenever $\mu \ne \nu$.
In what follows, we assume that $m,n \to \infty$ and $m/N \to \tau \in (0,1)$ with $N = m+n$.

The previous theorems will help us construct tests that enjoy asymptotic consistency and correct asymptotic level based on the smooth $p$-Wasserstein distance, using $W_{m,n}:= \sqrt{\frac{mn}{N}} \, \GWp(\hat{\mu}_m, \hat{\nu}_n)$.
Two-sample testing using the Wasserstein distance was previously explored in \cite{ramdas_two_sample_2017}, but these results are fundamentally restricted to the one-dimensional setting. Specifically, while the authors designed tests with data-independent critical values for $d=1$, they rely heavily on limit distributions of empirical Wasserstein distances that do not extend to higher dimensions. Our results use data-dependent critical values but scale to arbitrary dimension, demonstrating the compatibility of smooth distances for multivariate two-sample testing. %

We use the bootstrap to calibrate critical values. 
Consider the pooled data $(Z_1, \dots, Z_N) = (X_1, \dots, X_m, Y_1, \dots, Y_n)$ with empirical distribution $\hat{\gamma}_N = N^{-1}\sum_{i=1}^N \delta_{Z_i}$.
Let $X_1^B,\dots,X_m^B$ and $Y_1^B,\dots,Y_n^B$ be i.i.d. from $\hat{\gamma}_N$ given $Z_1, \dots, Z_N$, and take $\hat{\mu}_m^B = m^{-1}\sum_{i=1}^m \delta_{X_i^B}$ and $\hat{\nu}_n^B = n^{-1}\sum_{i=1}^n \delta_{Y_i^B}$ to be the corresponding bootstrap empirical measures. 

Specifying critical values requires a bit of care, as the comparison inequality \eqref{eq:duality} requires centering of one of measures. So we center the bootstrap empirical measures $\hat{\mu}_m^B$ and $\hat{\nu}_n^B$ by the pooled sample mean $\bar{Z} = N^{-1}\sum_{i=1}^{n}Z_i$, namely, we apply the bootstrap as
\[
W^{B}_{m,n} =p\, e^{\frac{\tr \hat{\Sigma}_Z}{2q \sigma^2}}  \sqrt{\frac{mn}{N}}  \Gds\left(\hat{\mu}_m^B*\delta_{-\bar{Z}_N}, \hat{\nu}_n^B*\delta_{-\bar{Z}_N}\right),
\]
where $\hat{\Sigma}_Z = N^{-1}\sum_{i=1}^{N}(Z_i-\bar{Z}_N)(Z_i-\bar{Z}_N)^{\top}$. Denote~the conditional $(1-\alpha)$-quantile of $W_{m,n}^{B}$ by $w_{m,n}^{B}(1-\alpha)$, i.e.,
\[
w_{m,n}^{B}(1-\alpha) = \inf \left \{ t : \Prob^{B}(W_{m,n}^{B} \le t) \ge 1-\alpha \right\}. 
\]
Then, we have the following result.

\begin{figure*}[t!]
\centering
\includegraphics[scale=0.65]{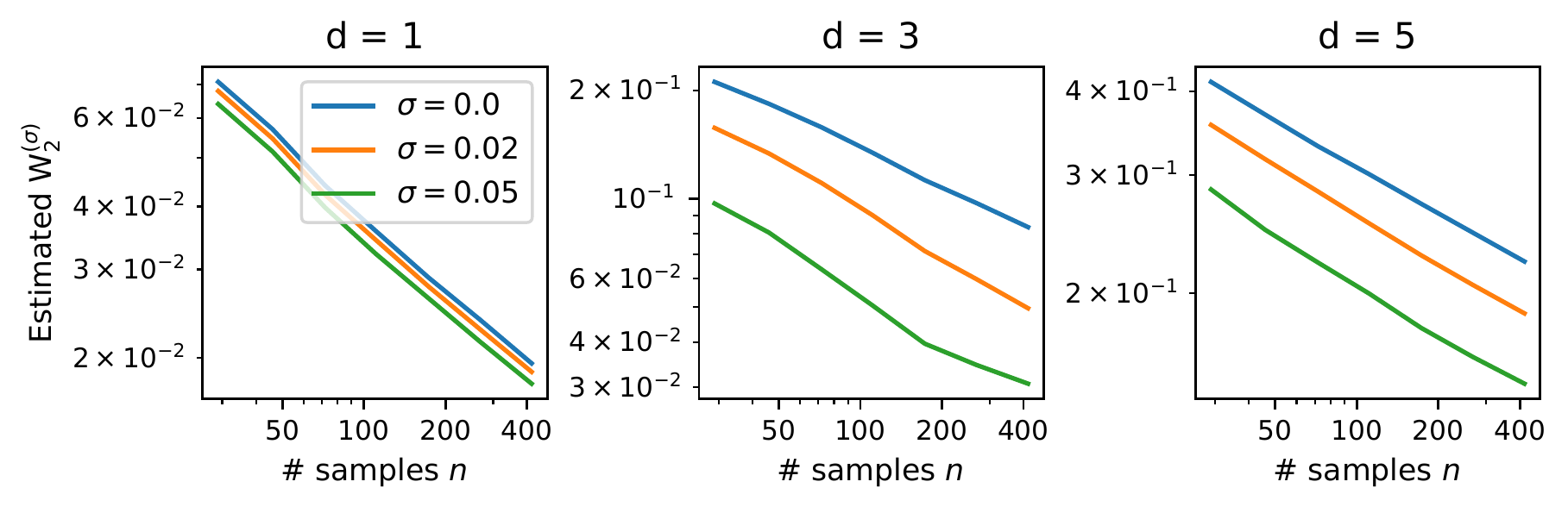}
\ \, \ \includegraphics[scale=0.641875]{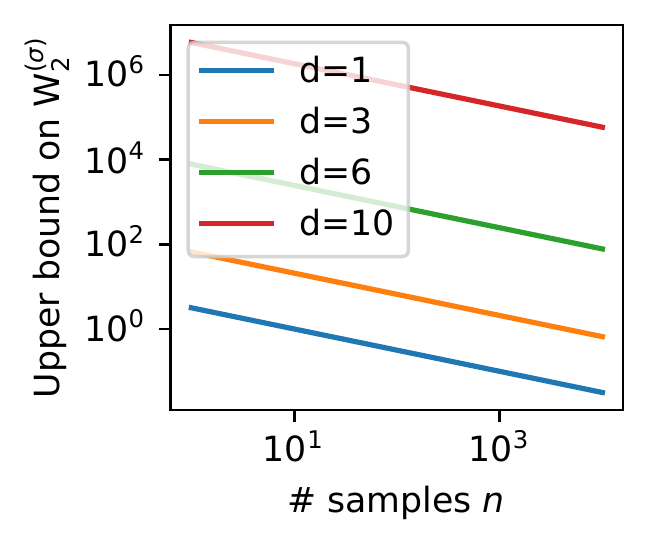}
\vspace{-4mm}
\caption{(Left) Empirical convergence of estimated $\E[\mathsf{W}^{(\sigma)}_2(\hat{\mu}_n,\mu)]$ to 0 for $\mu = \Unif([-1,1]^d)$ with $d \in \{1,3,5\}$. (Right) Loose upper bound on $\E[\mathsf{W}^{(0.5)}_2(\hat{\mu}_n,\mu)]$ provided by $\mathsf{d}_2^{(0.5)}$ for $\mu = \Unif([-1,1]^d)$ with $d \in \{1,3,6,10\}$.}\label{fig:empirical-convergence}
\end{figure*}

\begin{proposition}[Asymptotic validity]
\label{prop:two-sample_SW}
For $\mu,\nu \in \cP_p$ satisfying the condition of \cref{thm:limit-distribution}, the sequence of tests that reject the null hypothesis $H_0:\mu = \nu$ if $W_{m,n} > w^{B}_{m,n}(1-\alpha)$ is asymptotically consistent with  level $\alpha$.
\end{proposition}

We remark that the same argument gives a simpler result when $p=1$. Because $\mathsf{W}_1^{(\sigma)}$ is an IPM itself w.r.t.\ a function class that is $\mu$- and $\nu$-Donsker under moment conditions (see \citep[Theorem 1]{goldfeld2020asymptotic}), no centering is necessary, and $W_{m,n}^B$ can be replaced by $\sqrt{\frac{mn}{N}}\mathsf{W}_1^{(\sigma)}(\hat{\mu}^B_m,\hat{\nu}^B_n)$.

\subsection{Generative Modeling}

In the unsupervised learning task of generative modeling, we obtain an i.i.d.\ sample $X_1, \dots, X_n$ from a distribution $\mu \in \cP$ and aim to learn a generative model from a parameterized family $\{\nu_\theta\}_{\theta \in \Theta} \subset \cP$ which approximates $\mu$ under some statistical distance. We adopt the smooth Wasserstein distance as the figure of merit and use the empirical distribution $\hat{\mu}_n$ as an estimate for $\mu$. Generative modeling is thus formulated as the following minimum smooth Wasserstein estimation (M-SWE) problem:
\[
\inf_{\theta \in \Theta} \GWp(\hat{\mu}_n,\nu_\theta).
\]
In the unsmooth case, this objective with $\mathsf{W}_1$ inspired the Wasserstein GAN (W-GAN) framework that continues to underlie state-of-the-art methods in generative modeling \cite{arjovsky_wgan_2017,gulrajani2017improved}. M-SWE with $p=1$ was studied in \cite{goldfeld2020asymptotic}, and here we pursue similar measurability and consistency results for $p > 1$. 

In what follows, we take both $\mu \in \cP_p$ and $\{\nu_\theta\}_{\theta \in \Theta} \subset \cP_p$. Further, we suppose that $\Theta \subset \R^{d_0}$ is compact with nonempty interior and that $\theta \mapsto \nu_\theta$ is continuous w.r.t.\ the weak topology, i.e., $\nu_\theta \stackrel{w}{\to} \nu_{\bar{\theta}}$ whenever $\theta \to \bar{\theta}$. We start by establishing measurability, consistency, and parametric convergence rates for M-SWE.%

\begin{proposition}[M-SWE measurability]
\label{prop:M-SWE-measurability}
For each $n \in \N$, there exists a measurable function $\omega \mapsto \hat{\theta}_n(\omega)$ such that $\hat{\theta}_n(\omega) \in \argmin_{\theta \in \Theta} \GWp(\hat{\mu}_n(\omega),\nu_\theta)$.
\end{proposition}

\begin{proposition}[M-SWE consistency]
\label{prop:M-SWE_consistency}
The following hold:
\vspace{-3.5mm}
\begin{enumerate}[wide, labelindent=0pt]
    \item $\inf_{\theta \in \Theta} \GWp(\hat{\mu}_n, \nu_\theta) \to \inf_{\theta \in \Theta} \GWp(\mu,\nu_\theta)$ a.s.
    \vspace{-1.5mm}
    \item There exists an event with probability one on which the following holds: for any sequence $\{\hat{\theta}_n\}_{n \in \N}$ of measurable estimators such that $\GWp(\hat{\mu}_n,\nu_{\hat{\theta}_n}) \leq \inf_{\theta \in \Theta} \GWp(\hat{\mu}_n,\nu_\theta) + o_{\Prob}(1)$, the set of cluster points of $\{\hat{\theta}_n\}_{n \in \N}$ is included in $\argmin_{\theta \in \Theta} \GWp(\mu,\nu_\theta)$.
    \vspace{-1.5mm}
    \item If $\argmin_{\theta \in \Theta} \GWp(\mu,\nu_\theta) = \{\theta^\star\}$, then $\hat{\theta}_n \to \theta^\star$ a.s. 
\end{enumerate}
\end{proposition}

\begin{proposition}[M-SWE convergence rate]
If $\mu$ satisfies the conditions of \cref{thm:limit-distribution}, then \[\left|\inf_{\theta \in \Theta} \GWp(\hat{\mu}_n, \nu_\theta)-\inf_{\theta \in \Theta} \GWp(\mu,\nu_\theta)\right|=O_\PP(n^{-1/2}).\] 
\end{proposition}

Likewise, under additional regularity conditions, the solutions $\hat{\theta}_n$ to M-SWE converge at the parametric rate, i.e., $|\hat{\theta}_n - \theta^\star| = O_{\PP}(n^{-1/2})$, where $\theta^\star$ is the unique solution as above; see \cref{prf:applications} for details.

These propositions follow by similar arguments to those in \cite{goldfeld2020asymptotic}, which build on \cite{pollard_min_dist_1980}, with arbitrary $p\geq 1$ instead of  $p=1$ as considered therein (the needed results from \cite{villani2008optimal} hold for all $p \geq 1$). We thus omit their proofs for brevity. 

We next examine a high probability generalization bound for generative modeling via M-SWE, in accordance to the framework from \cite{arora_gans_2017,zhang_generalization_2018}. Thus, we want to control the gap between the $\GWp$ loss attained by approximate, possibly suboptimal, empirical minimizers and the population loss $\inf_{\theta\in\Theta}\GWp (\mu,\nu_\theta)$. Upper bounding this gap by the rate of empirical convergence, the concentration result \cref{prop:concentration1} implies the following.

\begin{corollary}[M-SWE generalization error]
\label{cor:M-SWE-generalization-error}
Assume $\mu$ has compact support and let $\hat{\theta}_n$ be an estimator with $\GWp(\hat{\mu}_n,\nu_{\hat{\theta}_n}) \leq \inf_{\theta \in \Theta} \GWp(\hat{\mu}_n,\nu_\theta) + \epsilon$, for some $\epsilon>0$. We have %
\[
\Prob\mspace{-2mu}\left(\mspace{-2mu}\GWp(\mu,\nu_{\hat{\theta}_n}) \mspace{-1mu}- \mspace{-2mu}\inf_{\theta \in \Theta}\GWp(\mu,\nu_\theta) \mspace{-1mu}>\mspace{-1mu} \mspace{-1mu}\epsilon \mspace{-1mu}+ \mspace{-1mu}t \mspace{-1mu}\right) \mspace{-2mu}\leq \mspace{-1mu}C e^{-c nt^2},
\]
for constants $C,c $ independent of $n$ and $t$.
\end{corollary}

\section{Experiments}
\label{sec:experiments}

We present several numerical experiments supporting the theoretical results established in the previous sections. 
We focus on $p=2$ so that we can use the MMD form for $\mathsf{d}_2^{(\sigma)}$.
Code is provided at \url{https://github.com/sbnietert/smooth-Wp}.

First, we examine $\mathsf{W}_2^{(\sigma)}(\mu,\hat{\mu}_n)$ directly, with computations feasible for small sample sizes using the stochastic averaged gradient (SAG) method as proposed by \cite{genevay2016} and implemented by \cite{hallin2020}. 
In Figure \ref{fig:empirical-convergence} (left), we take $\mu = \Unif([-1,1]^d)$ %
and estimate $\E[\mathsf{W}_2^{(\sigma)}(\hat{\mu}_n,\mu)]$ averaged over 10 trials, for varied $d$ and $\sigma$. We observe the contractive property of $\mathsf{W}_2^{(\sigma)}$ and the speed up in convergence rate due to smoothing. However, SAG is not well-suited for computation in high dimensions and larger values of $\sigma$. Indeed, this method computes standard $\mathsf{W}_2$ between the convolved measures, which needs an exponential in $d$ number of samples from the Gaussian measure. %
Recently, \citet{vacher21} suggested that OT distances between smooth densities (like those of our convolved measures) may be computed more efficiently, but their algorithm is restricted to compactly supported distributions and leaves important hyperparameters unspecified.

Turning to $\mathsf{d}_2^{(\sigma)}$, the MMD form from \eqref{eq:MMD} readily enables efficient computation. In Figure \ref{fig:empirical-convergence} (right), we plot the $n^{-1/2}$ upper bound it gives for $\mathsf{W}_2^{(\sigma)}$ empirical convergence, using a closed form for relevant expectations described in \cref{app:rough-upper-bound}. %
We emphasize that this bound is too loose for practical approximation and serves rather as a theoretical tool for obtaining correct rates.
In Figure \ref{fig:Gds-limit-distributions}, we plot distributions of $\sqrt{n}\,\mathsf{d}_2^{(1)}(\hat{\mu}_n,\mu)$ as $n$ increases, using $\mu = \cN_s$ with varied $s$ and $d=5$. %
Distributions are computed using kernel density estimation over 50 trials and estimating $\mu$ by $\hat{\mu}_{1000}$.
We see convergence to a limit distribution for small $\sigma$ (estimating $\|G \|_{\infty,\cF}$ from \cref{thm:limit-distribution}) and note the necessity of the sub-Gaussian condition \eqref{eq:moment-condition}, with convergence failing as $\sigma$ surpasses $1/\sqrt{2}$, supporting \cref{lem:subGaussian}.

\begin{figure}[t!]
\vspace{3mm}
\includegraphics[width=7.7cm]{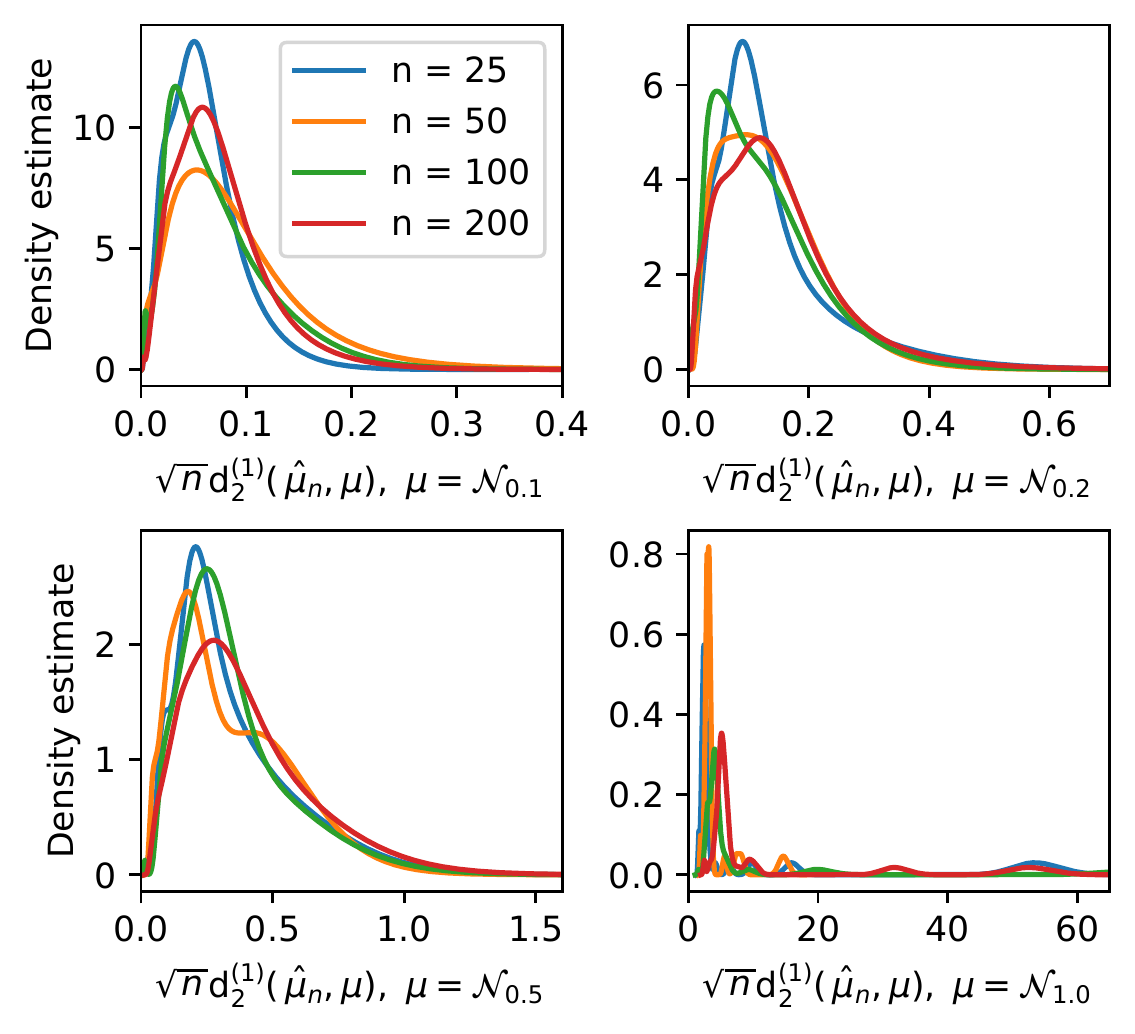}
\vspace{-1em}
\caption{Empirical limiting behavior of $\sqrt{n}\,\mathsf{d}_2^{(1)}(\hat{\mu}_n,\mu)$ for $\mu = \cN_s$ with $s \in \{0.1,0.2,0.5,1.0\}$ and $d = 5$.}
\label{fig:Gds-limit-distributions}
\end{figure}

\begin{figure}[t!]
\includegraphics[width=8cm]{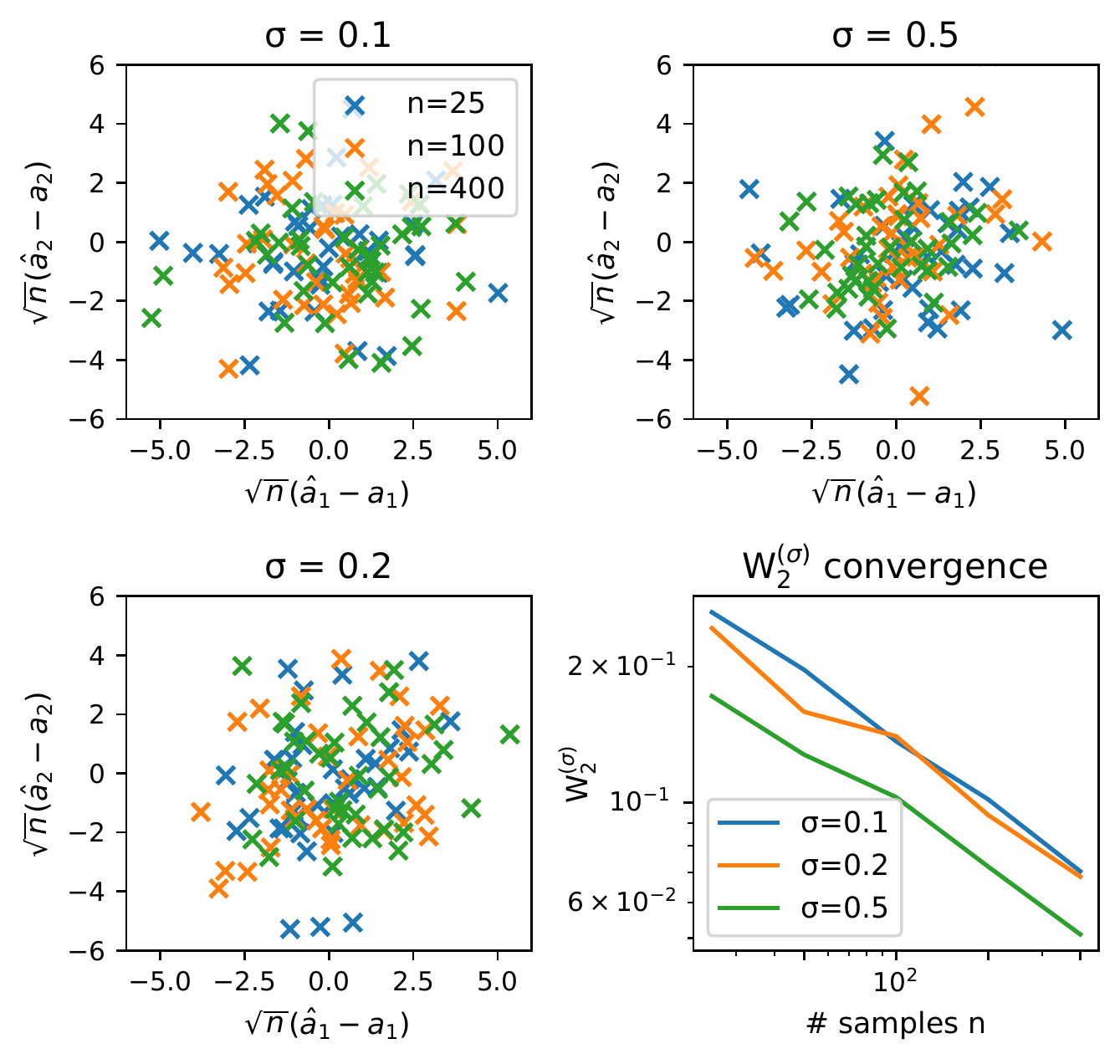}
\vspace{-1em}
\caption{One-dimensional limiting behavior of M-SWE estimates for the two mean parameters of $\mu = \cN(a_1,1)/2 + \cN(a_2,1)/2$ with $a_1 = -1$ and $a_2 = 1$. Also shown is a log-log plot of $\mathsf{W}_2^{(\sigma)}$ convergence in $n$.}
\label{fig:S-MWE}
\end{figure}

Next, we examine M-SWE when $d=1$, exploiting the fact that $\Wp$ can be expressed as an $L^p$ distance between quantile functions (see, e.g., \citep{villani2008optimal}).
The considered task is fitting a two-parameter generative model to a Gaussian mixture (parameterized by the means of its two modes). Distance minimization is implemented via gradient descent. Plotted in Figure \ref{fig:S-MWE} are $\sqrt{n}$-scaled scatter plots of the estimation errors, with 40 trials for each $\sigma$ and $n$ pair.  %
The consistent spread of the (scaled) estimation errors as $n$ increases demonstrates the $n^{-1/2}$ convergence rate. The bottom-right subplot shows $\mathsf{W}_2^{(\sigma)}$ estimation errors that further support the fast convergence. In \cref{app:experiments}, we provide additional results for a single Gaussian parameterized by mean and variance.

\begin{figure}[h]
\includegraphics[width=8cm]{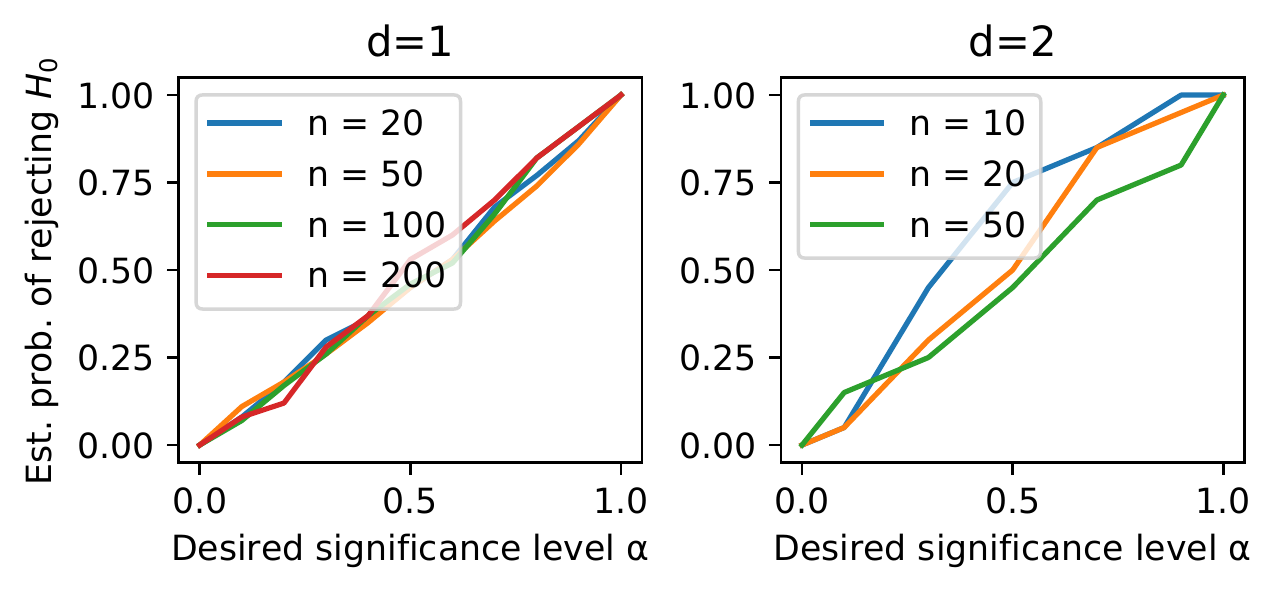}
\vspace{-1em}
\caption{Estimated probability that $\mathsf{W}_1^{(0.1)}$ two-sample test rejects null hypothesis $H_0:\mu = \nu$ given that $\mu = \nu = \Unif([0,1]^d)$.}
\label{fig:two-sample}
\end{figure}

Finally, we provide two-sample testing results in Figure \ref{fig:two-sample} for $p=1$, leveraging the simplifications discussed at the end of \cref{subsec:two-sample-testing}. We approximate the convolved empirical measures by adding Gaussian noise samples and compute $\mathsf{W}_1$ (exactly) for $d=1$ via its representation as the $L^1$ distance between quantile functions. For $d=2$, we estimate $\mathsf{W}_1$ using a standard implementation of W-GAN \cite{gulrajani2017improved, cao2017}. For varied sample sizes $n = m$, the quantiles of $\sqrt{\frac{n^2}{N}}\mathsf{W}_1^{(\sigma)}(\hat{\mu}^B_n,\hat{\nu}^B_n)$ are estimated using 1000 and 200 bootstrap samples for $d=1$ and $d=2$, respectively. The probability of rejecting the null hypothesis for varied significance levels and sample sizes is estimated by repeating the tests over 100 and 200 draws of the original samples, for $d=1$ and $d=2$ respectively. Figure \ref{fig:two-sample} displays the probability of false alarm versus the significance level $\alpha$. Evidently, the curves approximately fall along the diagonal $y=x$, supporting the consistency result.

\section{Conclusions and Future Directions}\label{sec:summary}

This work provided a thorough analysis of structural and statistical properties of the Gaussian-smoothed Wasserstein distance $\GWp$.
While $\GWp$ maintains many desirable properties of standard $\Wp$, we have shown via comparison to the smooth Sobolev IPM $\Gds$ that it admits a parametric empirical convergence rate, avoiding the curse of dimensionality that arises when estimating $\Wp$ from data.
Using this fast rate and the associated limit distribution for $\Gds$, we have explored new applications to two-sample testing and generative modeling.

An important direction for future research is efficient computation of $\GWp$. While standard methods for computing $\Wp$ are applicable in the smooth case (by sampling the noise), it is desirable to find computational techniques that make use of structure induced by the convolution with a known smooth kernel. Furthermore, while $\GWp$ exhibits an expected empirical convergence rate of $O(n^{-1/2})$ that is optimal in $n$, the prefactor scales exponentially with dimension and warrants additional study.
We suspect that this scaling can be shown under a manifold hypothesis to depend only on the intrinsic dimension of the data distribution rather than that of the ambient space.

Finally, we are interested in the limiting behavior of $\GWp$ as $\sigma \to 0$ and $p \to \infty$. The former case has implications for standard $\Wp$ and its dependence on intrinsic dimension, as well as for noise annealing that is common in machine learning practice. The latter may connect to differential privacy, where smoothing corresponds to (Gaussian) noise injection and $\mathsf{W}_\infty$ underlies the Wasserstein privacy mechanism \cite{song2017}.

\section*{Acknowledgements}
S. Nietert is supported by the National Science Foundation (NSF) Graduate Research Fellowship under Grant DGE-1650441. 
Z. Goldfeld is supported by the NSF CRII grant CCF-1947801, in part by the 2020 IBM Academic Award, and in part by the NSF CAREER Award CCF-2046018. 
K. Kato is partially supported by  NSF grants DMS-1952306 and DMS-2014636.

\bibliography{references}
\bibliographystyle{icml2021}

\clearpage

\appendix

\section*{Appendix}

\renewcommand{\thesection}{A.\arabic{section}}

\setcounter{section}{-1}
\section{Additional Notation}\label{sec:extra_notation}
For a given probability measure $\mu \in \cP$, let $\Phi_\mu(t) := \E\bigl[e^{i\langle t,X \rangle}\bigr]$ with $X \sim \mu$ denote its characteristic function.
Let $C^{k}(\R^{d})$ denote the class of $k$-times continuously differentiable functions on $\R^d$.
Let $\cL(X)$ denote the law of a random variable $X$.
We write $\lesssim$ for inequalities up to some numerical constant.

\section{Proofs for \cref{sec:smooth-Wp}}\label{sec:proofs3}
We first prove the following lemmas. 

\begin{lemma}[General smooth metrics]
\label{lem:general-smooth-metric}
Let $\kappa \in \cP$ be a distribution whose characteristic function never vanishes. %
If $\mathsf{d}$ is a metric on $\cX \subset \cP$ and $\cX$ is closed under taking convolutions with $\kappa$, then $\mathsf{d}_\kappa:(\mu,\nu) \mapsto \mathsf{d}(\mu*\kappa,\nu*\kappa)$ is also a metric on $\cX$.
\end{lemma}
\begin{proof}
Non-negativity and symmetry follow from definition.
The triangle inequality is also straightforward, since for $\mu_1,\mu_2,\mu_3 \in \cX$, the triangle inequality for $\mathsf{d}$ gives
\begin{align*}
\mathsf{d}_\kappa(\mu_1,\mu_2) &= \mathsf{d}(\mu_1*\kappa,\mu_2*\kappa)\\
&\leq \mathsf{d}(\mu_1*\kappa,\mu_3*\kappa) + \mathsf{d}(\mu_3*\kappa,\mu_2*\kappa)\\
&= \mathsf{d}_\kappa(\mu_1,\mu_3) + \mathsf{d}_\kappa(\mu_3,\mu_2).
\end{align*}
Finally, if $\mathsf{d}_\kappa(\mu,\nu) = 0$, then $\mu*\kappa = \nu*\kappa$.
Recalling that the characteristic function of a convolution of measures factors into a product, i.e., $\Phi_{\mu_1*\mu_2} = \Phi_{\mu_1} \Phi_{\mu_2}$, and since the characteristic function of $\kappa$ never vanishes, we have $\mu = \nu$. 
\end{proof}

\begin{lemma}[Contractive property of convolution]
\label{lem:contraction}
For any probability measure $\kappa \in \cP$, $\Wp(\mu*\kappa,\nu*\kappa) \leq \Wp(\mu,\nu)$. In particular, $\GWp(\mu,\nu) \leq \Wp(\mu,\nu)$.
\end{lemma}
\begin{proof}
Let $(X,Y)$ be an optimal coupling for $\Wp(\mu,\nu)$. Then taking $Z \sim \kappa$ independently, 
\begin{align*}
\Wp(\mu*\kappa,\nu*\kappa)^p &\leq \E\big[|(X+Z)-(Y+Z)|^p\big]\\
&=\E\big[|X-Y|^p\big] = \Wp(\mu,\nu).\qedhere
\end{align*}
\end{proof}

\begin{lemma}[Coupling decomposition]
\label{lem:smooth-Wp-coupling}
If $\pi \in \Pi(\mu*\Gauss,\nu*\Gauss)$, then there exists a coupling $(X,Y,Z,Z')$ such that $(X,Z) \sim \mu \otimes \Gauss$, $(Y,Z') \sim \nu \otimes \Gauss$, and $(X + Z,Y + Z') \sim \pi$.
\end{lemma}
\begin{proof}
If suffices to find a coupling $(X+Z,Y+Z',Z,Z')$ with the correct marginals. First, note that we already have couplings $(X+Z,Y+Z')$, $(X+Z,Z)$ and $(Y+Z',Z')$, given by $\pi$, $(\mu*\Gauss)\otimes \Gauss$, and $(\nu*\Gauss) \otimes \Gauss$, respectively. Hence, we can apply the gluing lemma (see, e.g., \cite{villani2003})
between $\pi$ and $(\mu*\Gauss)\otimes \Gauss$ to obtain a coupling $(X+Z,Y+Z',Z)$ and then between $\pi$, $(\nu*\Gauss)\otimes \Gauss$ to obtain a coupling $(X+Z,Y+Z',Z')$. We apply the gluing lemma a final time between the outcomes of its previous applications to obtain a coupling $(X+Z,Y+Z',Z,Z')$.
\end{proof}

\subsection{Proof of \cref{prop:smooth-wasserstein-metric}}
\label{prf:smooth-wasserstein-metric}

\cref{lem:general-smooth-metric} verifies that $\Wp^{(\sigma)}$ is a metric on $\cP_{p}$, since $\Phi_{\Gauss}(t) = e^{-\sigma^2|t|^2/2}\neq 0$, for all $t\in\R^d$.
To show that $\Wp^{(\sigma)}$ induces the same topology as $\Wp$, it suffices to prove that 
\[
\Wp(\mu_n,\mu) \to 0 \iff \Wp^{(\sigma)}(\mu_n,\mu) \to 0.
\]
The ``$\Rightarrow$" direction follows by \cref{lem:contraction}.
For the other direction, suppose that $\Wp^{(\sigma)}(\mu_n,\mu)\to0$. By \cref{lem:smooth-Wp-coupling}, we can find a coupling $\big((X_n,Z_n),(X,Z)\big)$ with $(X_n,Z_n) \sim \mu_n \otimes \Gauss$ and $(X,Z) \sim \mu \otimes \Gauss$ such that $\Wp^{(\sigma)} (\mu_n,\mu)^{p} = \E[|X_n+Z_n - (X+Z)|^p]$.
We will show that $X_{n} \stackrel{d}{\to} X$ and $\E[|X_{n}|^{p}] \to \E[|X|^{p}]$, which yields the desired result.

To that end, it is sufficient (and necessary) to show that $X_{n} \stackrel{d}{\to} X$ and that $|X_{n}|^{p}$ is uniformly integrable. 
Since convergence in distribution is equivalent to pointwise convergence of characteristic functions, from $X_{n}+Z_{n} \stackrel{d}{\to} X+Z$, we have for all $t \in \R^d$ that
\begin{align*}
\lim_{n\to\infty}\Phi_{\mu_n}(t)e^{-\sigma^2|t|^2/2}&=\lim_{n \to \infty}\Phi_{\mu_n*\Gauss}(t)\\
&= \Phi_{\mu*\Gauss}(t)=\Phi_\mu(t)e^{-\sigma^2|t|^2/2}, 
\end{align*}
implying that $\lim_{n \to \infty}\Phi_{\mu_n}(t)=\Phi_\mu(t)$, for all $t \in \R^{d}$, and hence that $X_n \stackrel{d}{\to} X$.
To verify the uniform integrability, observe that $|X_{n}|^{p} \le 2^{p-1}(|X_{n}+Z_{n}|^{p} + |Z_{n}|^{p})$.
By construction, $|X_{n}+Z_{n}|^{p}$ is uniformly integrable, while $|Z_{n}|^{p} \stackrel{d}{=} |Z|^{p}$ is trivially uniformly integrable, implying the uniform integrability of their sum and hence $|X_{n}|^{p}$.
\qed

\subsection{Proof of \cref{lem:stability}}
\label{prf:stability}
By \cref{lem:contraction}, we have $\mathsf{W}_p^{(\sigma_2)}(\mu,\nu) \leq \mathsf{W}_p^{(\sigma_1)}(\mu,\nu)$. For the other direction, let $X \sim \mu$, $Y \sim \nu$,  $Z_X \sim \cN_{\sigma_1}$, $Z_Y \sim \cN_{\sigma_1}$, $Z_X' \sim \cN_{\sqrt{\sigma_2^2 - \sigma_1^2}}$, and $Z_Y' \sim \cN_{\sqrt{\sigma_2^2 - \sigma_1^2}}$.
The smooth $p$-Wasserstein distance of parameter $\sigma_2$ is given as a minimization over couplings of the aforementioned random variables subject to the mutual independence of $(X, Z_X, Z_X')$ along with that of $(Y, Z_Y, Z_Y')$.
With this convention, we have
\begin{multline*}
\Wp^{(\sigma_2)}(\mu,\nu)\\ = \inf \Big (\E\Big[ \big|\big( (X + Z_X) - (Y+Z_{Y}) \big) + (Z_X'- Z_Y')\big|^p\Big] \Big)^{1/p}.
\end{multline*}
Now, Minkoski's inequality gives
\begin{align*}
\Wp^{(\sigma_2)}(\mu,\nu)
&\ge \inf \Bigl[ \Big(\E \Big[\big|(X + Z_X) - (Y + Z_Y)\big|^p\Big]\Big)^{1/p}\\
&\quad\qquad - \Big(\E\Big[\big|Z_X' - Z_Y'\big|^p\Big]\Big)^{1/p} \,\Bigr] \\
&\ge \Wp^{(\sigma_1)}(\mu,\nu) - \sup \big(\E \big[|Z_X' - Z_Y'|^p\big]\big)^{1/p} \\
&\ge \Wp^{(\sigma_1)}(\mu,\nu) - 2  \big(\E \big[|Z_X'|^p\big]\big)^{1/p}.
\end{align*}
Recall that for $Z \sim \cN (0,\mathrm{I}_{d})$, 
\[
\E\big[|Z|^{p}\big] = \frac{2^{p/2} \Gamma((p+d)/2)}{\Gamma(d/2)}.
\]
If $p$ is even, then above term is bounded by $(d + 2p - 2)^{p/2}$. In general, we round $p$ up to the nearest even integer to obtain the bound $(d + 2p + 2)^{p/2}$, completing the proof.
\qed

\subsection{Proof of \cref{prop:smooth-Wp-dependence-on-sigma}}
\label{prf:smooth-Wp-dependence-on-sigma}
The proof follows that of Theorem 3 in \cite{Goldfeld2020GOT}.
For Claim (ii), we simply apply %
Lemma 1, taking $\sigma_1 = 0$ and $\sigma_2 \to 0$.
For Claim (i), monotonicity follows directly from the contractive property established in the previous proof. For left continuity of $\GWp$, we apply %
Lemma 1 with $\sigma_2 = \sigma$ and $\sigma_1 \nearrow \sigma$. For right continuity, take $\sigma_k \searrow \sigma$ and define $\eps_k = \sqrt{\sigma_k^2 - \sigma^2}$. Then,
\begin{align*}
\mathsf{W}_p^{(\sigma_k)}(\mu,\nu) = \mathsf{W}_p^{(\eps_k)}(\mu*\Gauss,\nu*\Gauss) \to \GWp(\mu,\nu) 
\end{align*}
as $k \to \infty$. Claim (iii) follows from Corollary 2.4 of \cite{chen2020}.
\qed

\subsection{Proof of \cref{prop:convergence-of-plans}}
\label{prf:convergence-of-plans}
A close inspection of the proof of Theorem 4 in \cite{Goldfeld2020GOT}, which covers the $p=1$ case up to extraction of a subsequence, reveals that the only required properties of $|\cdot|^1$ are its non-negativity and continuity. These also hold for $|\cdot|^p$, so the theorem applies to $\GWp$. Further, the proof implies that any weakly convergent subsequence of couplings converges to an optimal coupling for $\GWp(\mu,\nu)$. When $p > 1$ and $\sigma > 0$, optimal couplings are unique (see, e.g., Theorem 2.44 of \cite{villani2003}), so Prokhorov's Theorem implies that extraction of a subsequence is not necessary.\qed

\subsection{Proof of \cref{thm:smooth-Wp-Sobolev-comparison}}
\label{prf:smooth-Wp-Sobolev-comparison}

We begin with a useful result bounding unsmoothed $\Wp$ by a dual Sobolev norm, adapting a proof from \cite{dolbeault2009}.

\begin{lemma}
\label{lem:comp-ub}
Fix $p > 1$ and suppose that $\mu_0, \mu_1 \in \cP_p$ with $\mu_0,\mu_1 \ll \gamma$ for some locally finite Borel measure $\gamma$ on $\R^d$. Denote their respective densities by $f_i = \dd \mu_i/\dd \gamma$. 
If $f_0$ or $f_1$ is lower bounded by some $c>0$, then we have
\[
\Wp(\mu_0,\mu_1) \leq p\, c^{-1/q} \| \mu_0 - \mu_1 \|_{\dot{H}^{-1,p}(\gamma)}.
\]
\end{lemma}
\begin{proof}
We essentially apply Theorem 5.26 of \cite{dolbeault2009}, which (for the choice of $\phi(\rho,w) = \rho^{1-p}|w|^p$), bounds $\Wp$ from above by the relevant dual Sobolev norm times a constant which depends on a lower bound for \emph{both} $f_0$ and $f_1$. The proof exploits the dynamic Benamou-Brenier formulation of optimal transport and the path in $(\cP_p,\Wp)$ which interpolates linearly between densities. Before concluding, they show
\[
\Wp(\mu_0,\mu_1)^p \leq \int_0^1 \int_{\R^d} ((1-t)f_0 + tf_1)^{1-p} |w|^p \, \dd \gamma\, \dd t,
\]
where $\|w\|_{L_p(\gamma;\R^d)} = \|\mu_0 - \mu_1\|_{\dot{H}^{-1,p}(\gamma)}$ (such $w$ is shown to exist only assuming $\|\mu_0 - \mu_1\|_{\dot{H}^{-1,p}(\gamma)} < \infty$).
However, even with the lower bound $c$ on just one of the densities (say $f_0$ without loss of generality), we have
\begin{align*}
&\int_0^1 \int_{\R^d} ((1-t)f_0 + tf_1)^{1-p} |w|^p \, \dd \gamma\, \dd t\\ \leq &\int_0^1 (tc)^{1-p} \int_{\R^d}  |w|^p \, \dd \gamma\, \dd t\\
&= c^{1-p} \|w\|_{L_p(\gamma;\R^d)}^p \int_0^1 t^{1-p}\, \dd t\\
&= p^p c^{1-p} \|\mu_0 - \mu_1\|_{\dot{H}^{-1,p}(\gamma)}^p,
\end{align*}
which gives the lemma.
\end{proof}

To prove the theorem, we apply the lemma with $\mu_0 = \mu * \Gauss$, $\mu_1 = \nu * \Gauss$, and $\gamma = \Gauss$. To bound $\dd \mu * \Gauss / \dd \Gauss$ from below, let $X \sim \mu$ and compute
\[
\begin{split}
\mu*\varphi_{\sigma}(y) &= \frac{1}{(2\pi\sigma^2)^{d/2}} \int_{\R^{d}} e^{-|x-y|^2/(2\sigma^2)} \dd\mu(x) \\
&\ge \frac{1}{(2\pi \sigma^2)^{d/2}} e^{-\E[|y-X|^2/(2\sigma^2)]},
\end{split}
\]
where the second step uses Jensen's inequality. The desired conclusion follows because $\E[|y-X|^2] = |y|^2 + \E[|X|^2] - 2\langle y, \E[X] \rangle$ and $X$ has mean zero.

For a related lower bound, we will apply Theorem 5.24 of \cite{dolbeault2009} with the choice of $\phi(\rho,w) = |w|^p$ to see that $\Wp(\mu_0,\mu_1) \geq C^{-1} \| \mu_0 - \mu_1 \|_{\dot{H}^{-1,p(\gamma)}}$ under the same conditions as \cref{lem:comp-ub} but where $C$ is now an upper bound on the densities. To start, we compute
\begin{equation*}
\begin{split}
\frac{\mu*\varphi_{\sigma}(y)}{\varphi_{\sqrt{2}\sigma}(y)} &= 2^{d/2} \int_{\R^{d}} e^{-\frac{|y-x|^2}{2\sigma^2} + \frac{|y|^2}{4\sigma^2}} \dd\mu(x) \\
&= 2^{d/2} \int_{\R^{d}} e^{-\frac{|y-2x|^2}{4\sigma^2} + \frac{|x|^2}{2\sigma^2}} \dd\mu(x)\\
&\leq 2^{d/2} \E\big[e^{|X|^2/(2\sigma^2)}\big],
\end{split}
\end{equation*}
where $X \sim \mu$. Hence,
\begin{align*}
    \GWp(\mu_0,\mu_1) &\geq 2^{-d/2}\\ &\left(\E\Big[e^{|X_0|^2/(2\sigma^2)}\Big] \land \E\Big[e^{|X_1|^2/(2\sigma^2)}\Big] \right)^{-1}\\
    &\:\big\|(\mu_0 - \mu_1)*\cN_{\sigma}\big\|_{\dot{H}^{-1,p}(\cN_{\sqrt{2}\sigma})},
\end{align*}
where $X_0 \sim \mu_0$ and $X_1 \sim \mu_1$. This bound is only meaningful when $\mu_0$ and $\mu_1$ are sufficiently sub-Gaussian.
\qed

\subsection{Proof of \cref{prop:Gds-sigma-dependence}}
For (i), we observe that if $\mu \neq \nu$, then the two measures must share a continuity set $A$ such that $\mu(A) \neq \nu(A)$. We can assume without loss of generality that $A$ does not contain the origin and that $(\mu - \nu)(A) > 0$. Then, for any $C > 0$, there exists sufficiently small $\sigma$ such that
\begin{align*}
    \Gds(\mu,\nu) &= \sup_{\substack{f: \|\nabla f\|_{L^q(\Gauss)} \leq 1}} (\mu*\Gauss - \nu*\Gauss)(f)\\
    &\geq (\mu*\Gauss - \nu*\Gauss)(C\mathds{1}_A)\\
    &= C(\mu*\Gauss - \nu*\Gauss)(A)\\
    &\geq \frac{C}{2}(\mu - \nu)(A).
\end{align*}
By taking $C$ arbitrarily large, we see that $\Gds(\mu,\nu) = \infty$, establishing (i).
For (ii), we employ \cref{thm:dual-Sobolev-RKHS} and observe that \begin{equation*}
    \kappa^{(\sigma)}(x,y) = \langle x,y \rangle + \frac{1}{4\sigma^2}\langle x,y \rangle^2 + O(\sigma^{-4}).
\end{equation*}
As $\sigma \to \infty$, we obtain the pointwise limit kernel $\kappa^{(\infty)} = \langle x,y \rangle$, which induces the distance given in (ii). Swapping the limit and the expectation in \eqref{eq:MMD} is justified by the Dominated Convergence Theorem given that $\mu$ and $\nu$ are sub-Gaussian.
\qed

\section{Proofs for \cref{sec:empirical}}

\subsection{Proof of \cref{thm:slow-rate}}
\label{prf:slow-rate}
The argument relies on Proposition 7.10 from \cite{villani2003}, which is restated~next.

\begin{lemma}[Proposition 7.10 in \cite{villani2003}]
For any $1 \leq p < \infty$, we have
\begin{equation}
\label{eq:weighted-TV}
\Wp (\mu,\nu) \leq 2^{\frac{p-1}{p}}\left(\int_{\R^d} |x|^{p} \dd|\mu-\nu|(x)\right)^{1/p}.
\end{equation}
\end{lemma}

This bound follows by coupling $\mu$ and $\nu$ via the maximal TV-coupling and evaluating the resulting transportation cost.
Invoking the lemma and Jensen's inequality, we have
\begin{align*}
&\E\left[\GWp(\hat{\mu}_n,\mu)\right]\\
&\leq 2^{\frac{p-1}{p}} \left(\int_{\R^d}|x|^p\E\big[\big|\hat{\mu}_{n}*\varphi_\sigma(x)-\mu*\varphi_\sigma(x)\big|\big] \dd x\right)^{1/p}\\
&\leq 2^{\frac{p-1}{p}} n^{-\frac{1}{2p}}\left(\int_{\R^d}|x|^p\sqrt{\Var\big[\varphi_\sigma(x-X)\big]} \dd x\right)^{1/p},
\end{align*}
where the last inequality follows because $\E\big[\varphi_\sigma(x-X)\big]=\mu*\varphi_\sigma(x)$ for all $x \in \R^d$.
Furthermore,
\begin{align*}
&\Var\big[\varphi_\sigma(x-X)\big] \le \E[\varphi_{\sigma}(x-X)^2]\\
&=\frac{1}{(2\pi \sigma^2)^d} \int_{\R^{d}} e^{-\frac{|x-y|^2}{\sigma^2}} \dd \mu(y) \\
&=\frac{1}{(2\pi\sigma^2)^d} \left (\int_{|y| \le \frac{|x|}{2}} + \int_{|y| > \frac{|x|}{2}} \right) e^{-\frac{|x-y|^2}{\sigma^2}} \dd \mu(y) \\
&\le \frac{1}{(2\pi \sigma^2)^d} \left ( \int_{|y|\le \frac{|x|}{2}} e^{-\frac{|x-y|^2}{\sigma^2}} \dd \mu(y) + \Prob \left(|X| > \frac{|x|}{2}\right)\mspace{-5mu} \right ).
\end{align*}
If $|y| \le |x|/2$, then $|x-y|^2 \ge |x|^2/4$, which yields
\[
\sqrt{\Var(\varphi_{\sigma}(x-X))} \\
\le \frac{e^{-\frac{|x|^{2}}{8\sigma^2}} + \sqrt{\Prob\left(|X| > \frac{|x|}{2} \right)}}{(2\pi \sigma^2)^{d/2}}. 
\]
Direct calculations show that 
\[\int_{\R^{d}} |x|^{p} e^{-\frac{|x|^{2}}{8\sigma^2}} \dd x = \frac{8^{\frac{d+p}{2}}\sigma^{d+p}\pi^{d/2} \Gamma((d+p)/2)}{\Gamma(d/2)}\]
and
\begin{align*}
&\int_{\R^d} |x|^{p} \sqrt{\Prob\big(|X| > |x|/2\big)} \dd x \\
= \: &\frac{2^{d+p+1}\pi^{d/2}}{\Gamma (d/2)} \int_{0}^{\infty} r^{d+p-1}\sqrt{\Prob(|X| > r)} \dd r. 
\end{align*}
Hence $\E\Big[\Wp^{(\sigma)}(\hat{\mu}_n,\mu)\Big] = O\big(n^{-1/(2p)}\big)$ if Condition (2)
holds.
The last assertion follows from Markov's inequality. 

To specify the exact constant, we combine the above bounds and simplify to obtain to obtain
\begin{align*}
&\E\left[\GWp(\hat{\mu}_n,\mu)\right] \leq 2^{1-1/p}n^{-1/2p}\\
&\quad\left( \frac{2^{d+3p/2}\sigma^p \Gamma((d+p)/2)}{\Gamma(d/2)} + \frac{2^{d/2+p+1}I}{\Gamma(d/2)\sigma^d}\right)^{1/p},
\end{align*}
where $I$ is the integral from Condition (2). By the subadditivity of $t \mapsto t^{1/p}$ and properties of the gamma function, we bound the RHS above by
\begin{align*}
&8 n^{-1/2p} \left( 2^{d/p}\sigma \sqrt{d/2 + p + 1} + \frac{2^{d/(2p)}I^{1/p}}{\Gamma(d/2)^{1/p}\sigma^{d/p}}\right).
\end{align*}
If $\mu$ is $\beta$-sub-Gaussian, then $\Prob(|X| > r) \leq 2^{d/2}e^{-\frac{r^2}{4\beta^2}}$ and we can bound
\begin{align*}
   I &= \int_0^\infty r^{d+p-1}\sqrt{\Prob(|X| > r)}\, \dd r\\
   &\leq 2^{d/4} \int_0^\infty r^{d+p-1} e^{-r^2/(4\beta^2)} \,\dd r\\
   &= 2^{d/4-1}(2\beta)^{d+p} \Gamma((d+p)/2)\\
   &= 2^{5d/4+p-1} \beta^{d+p} \Gamma((d+p)/2).
\end{align*}
Plugging this into the previous bound, using properties of the gamma function, and simplifying, we obtain
\begin{align*}
&\E\left[\GWp(\hat{\mu}_n,\mu)\right] \leq 8 n^{-1/2p}\\
&\left( 2^{d/p}\sigma \sqrt{d + p} + 2^{7d/(4p)}\beta^{d/p+1}\sqrt{d+p} \sigma^{-d/p} \right)\\
&\leq 8 \cdot 4^{d/p} \sqrt{d+p} \left[ \sigma + \beta \left(\frac{\beta}{\sigma}\right)^{d/p} \right] \cdot n^{-1/(2p)}.
\end{align*}
\qed

\subsection{Proof of \cref{thm:limit-distribution}}
\label{prf:limit-distribution}
For $p \geq 1$, a probability measure $\gamma \in \cP$ is said to satisfy the \textit{$p$-Poincar\'{e} inequality} if there exists a finite constant $D$ such that 
\begin{equation}
\| f - \gamma(f) \|_{L^{p}(\gamma)} \le D \| \nabla f \|_{L^{p}(\gamma; \R^d)}, \quad \forall f \in C_{0}^{\infty}.
\label{eq:Poincare}
\end{equation}
The smallest constant satisfying the above is denoted by $D_{p}(\gamma)$.
We note in particular that $\Gauss$ satisfies a $p$-Poincar\'e inequality for all $p \geq 1$ (see, e.g., \cite{boucheron_concentration_2013} and Theorem 2.4 of \cite{milman2009}).

Let $\partial_{j} = \partial/\partial x_{j}$.
For any multi-index $k = (k_1,\dots,k_d) \in \N_{0}^{d}$, define the differential operator
\[
\partial^k = \partial_1^{k_1} \cdots \partial_d^{k_d},
\]
and let $\bar{k} = \sum_{j=1}^{d}k_j$. We start by bounding the derivatives of centered functions with bounded homogeneous Sobolev norm after Gaussian smoothing.

\begin{lemma}
\label{lem:derivative}
Fix $\eta > 0$. 
Pick any $f \in C_{0}^{\infty}$ such that $\|f\|_{\dot{H}^{1,q}(\cN_{\sigma})} \le 1$, and let $f_{\sigma} = f*\varphi_{\sigma} - \Gauss(f)$.
Then for any multi-index $k = (k_1,\dots,k_d) \in \N_{0}^{d}$,
\[
|\partial^{k} f_{\sigma}(x)| \lesssim (D_q(\cN_{\sigma}) \vee \sigma^{-\bar{k}+1})
\exp \left ( \frac{(p-1)(1+\eta)|x|^2}{2\sigma^2} \right)
\]
up to constants independent of $f, x$, and $\sigma$. 
\end{lemma}

\begin{proof}[Proof of Lemma \ref{lem:derivative}]
Observe that 
\begin{align*}
f_{\sigma}(x) &= \int \varphi_{\sigma}(x-y) f(y) \dd y\\
&= \int \frac{\varphi_{\sigma}(x-y)}{\varphi_{\sigma}(y)} f(y) \varphi_{\sigma}(y) \dd y. 
\end{align*}
Applying H\"{o}lder's inequality, we have 
\[
|f_{\sigma}(x)| \le \left [\int \frac{\varphi_{\sigma}^{p}(x-y)}{\varphi_{\sigma}^{p-1}(y)} \dd y \right ]^{1/p} \| f \|_{L^q (\cN_{\sigma})}.
\]
Here, since $\| \nabla f \|_{L^q(\cN_{\sigma};\R^d)} =\| f \|_{\dot{H}^{1,q}(\cN_{\sigma})} \le 1$, we have 
\[
\| f \|_{L^q(\cN_{\sigma})} \le D_q(\cN_{\sigma}) \| \nabla f \|_{L^q(\cN_{\sigma};\R^d)} \le D_q(\cN_{\sigma}). 
\]
Observe that 
\begin{align*}
&\int \frac{\varphi_{\sigma}^{p}(x-y)}{\varphi^{p-1}_{\sigma}(y)} \dd y\\ &= \frac{1}{(2\pi \sigma^2)^{d/2}}  \int \exp \left [ -\frac{p |x-y|^2 - (p-1)|y|^2 }{2\sigma^2}\right ] \dd y \\
&=e^{-p|x|^2/(2\sigma^2)} \int e^{p\langle x,y \rangle/\sigma^2}  \varphi_{\sigma}(y) \, \dd y \\
&=  \exp \left ( \frac{p(p-1)|x|^2}{2\sigma^2} \right ).
\end{align*}
This yields that 
\[
|f_{\sigma}(x)| \le  D_q(\cN_{\sigma}) \exp \left ( \frac{(p-1)|x|^2}{2\sigma^2} \right ),
\]
establishing the claim when $\bar{k} = 0$.

Next, we note that 
\begin{align*}
\nabla f_{\sigma}(x) &= \int [\nabla_{x} \varphi_{\sigma}(x-y) ] f(y) \dd y\\ &= -\int [\nabla_{y} \varphi_{\sigma}(x-y) ] f(y) \dd y\\
&=\int \varphi_{\sigma}(x-y) \nabla_{y} f(y) \dd y.
\end{align*}
Since $\| \nabla f \|_{L^{q}(\cN_\sigma;\R^d)} \le 1$, we can apply the preceding argument to conclude that 
\[
|\nabla f_{\sigma}(x)| \le \exp \left ( \frac{(p-1)|x|^2}{2\sigma^2} \right ).
\]

Finally, we extend to arbitrary derivatives, observing that for any $i=1,\dots,d$ and $k \in \mathbb{N}_{0}^{d}$,
\begin{align}
\label{eq:gauss-deriv-bound}
\begin{split}
&\partial^k \partial_i f_{\sigma}(x)\\ &= \int [\partial_i f(y)] \varphi_{\sigma}(x\mspace{-2mu}-\mspace{-2mu}y) \mspace{-5mu} \prod_{j=1}^{d} (-1)^{k_{j}} \sigma^{-k_{j}} \mspace{-3mu} \He_{k_{j}}\mspace{-5mu}  \left (\mspace{-5mu}\frac{x_{j}-y_{j}}{\sigma} \mspace{-5mu}\right)\mspace{-5mu} \dd y.
\end{split}
\end{align}
Here, we use that
\[
\partial^k \varphi_\sigma(z) = \varphi_{\sigma}(z) \left [ \prod_{j=1}^{d} (-1)^{k_{j}} \sigma^{-k_{j}} \He_{k_{j}} \big (z_j/\sigma \big)\right ],
\]
where $\He_n$ is the Hermite polynomial of degree $n$ defined by
\[
\He_n(x) = (-1)^n e^{x^2/2} \dv[n]{x} e^{x^2/2}.
\]
Return to \eqref{eq:gauss-deriv-bound}. Pick any $\eta > 0$. Since the product term in \eqref{eq:gauss-deriv-bound} can be bounded (up to constants) by $1 + |x - y|^{\bar{k}}$, we have
\[
|\partial^k \partial_j f_{\sigma}(x)| \lesssim \sigma^{-\bar{k}} \int |\partial_j f(y)| \: \varphi_{\sigma(1+\eta)^{-1/2}}(x-y) \dd y.
\]
up to a constant independent of $f,x$,and $\sigma$. The desired bound follows by the same argument we applied to control $|\nabla f_\sigma(x)|$.

Now, to be more precise with constants, we note that since $D_2(\Gauss) = \sigma^2$ and $\Gauss$ is log-concave, we have by Theorem 2.4 of \cite{milman2009} that $D_q(\Gauss) \leq C \sigma^2$ for all $q \in [1,\infty]$, for some absolute constant $C > 0$ . Next, we recall the explicit formula
\[
\He_n(x) = n! \sum_{m=0}^{\lfloor n/2 \rfloor} \frac{(-1)^{m}}{m!(n-2m)!} \frac{x^{n-2m}}{2^m}.
\]
Using $|x|^m \leq 1 + |x|^n$ for $m = 1, \dots, n$, we (quite loosely) bound
\[
|\He_n(x)| \leq n! (1 + |x|^n) \sum_{m=0}^{\lfloor n/2 \rfloor} \frac{1}{m!(n-2m)!2^m}.
\]
This summand is unimodal and attains its maximum at $m = \big\lceil \frac{n}{2} - \frac{\sqrt{n+2}}{2} \big\rceil$. Using this and Stirling's approximation, we find
\begin{align*}
|\He_n(x)| &\leq \frac{n! (1 + |x|^n) (n + 4)}{\Gamma\big( \frac{n}{2} - \frac{\sqrt{n+2}}{2} \big) \Gamma\big(\sqrt{n+2} - 1\big) 2^{\frac{n}{2} - \frac{\sqrt{n+2}}{2}}}\\
&\leq (1 + |x|^n) (cn)^{n/2},
\end{align*}
for some absolute constant $c > 0$. Now, the product term in \eqref{eq:gauss-deriv-bound} is bounded in absolute value by
\begin{align*}
\sigma^{-\bar{k}} \prod_{j=1}^d (1 + |z_j|)^n(c k_j)^{k_j/2} \leq \sigma^{-\bar{k}} (c \bar{k})^{\bar{k}/2}(1 + |z|)^{\bar{k}}.
\end{align*}
With a bit of calculus, we compute
\begin{align*}
    |\partial^k \partial_j f_{\sigma}(x)| \leq \: & (c' \bar{k})^{\bar{k}} \sigma^{-{\bar{k}}}(1 + \eta)^{d/2}\\
    &\int |\partial_j f(y)| \: \varphi_{\sigma(1+\eta)^{-1/2}}(x-y)  \dd y,
\end{align*}
for some second constant $c' > 0$ and any $\eta > 0$, so long as $\sigma \leq 1$, say. Applying the same argument used to control $|\nabla f_\sigma(x)|$, we bound
\begin{align*}
    &\int \frac{\varphi_{\sigma(1+\eta)^{-1/2}}(x-y)^p}{\varphi_\sigma(y)^{p-1}} \, \dd y \leq \\
    &(1 + p\eta)^d e^{-\frac{p(1+\eta)|x|^2}{2\sigma^2}} \int e^{-\frac{p(1+\eta)\langle x,y\rangle}{\sigma^2}} \varphi_{\sigma(1+\eta p)^{-1/2}}(y) \, \dd y\\
    &= (1 + p\eta)^d e^{-\frac{p(1+\eta)|x|^2}{2\sigma^2} + \frac{p^2(1 + \eta)^2|x|^2}{(1+\eta p)\sigma^2}}\\
    &= (1 + p\eta)^d e^{-\frac{p(1+\eta)|x|^2}{2\sigma^2} + \frac{p^2(1 + \eta)^2|x|^2}{(1+\eta p)\sigma^2}},
\end{align*}
which yields
\begin{align*}
    |\partial^k \partial_j f_{\sigma}(x)| \leq \,& (c' \bar{k})^{\bar{k}}\eta^{-\bar{k}/2} \sigma^{-{\bar{k}}}(1 + \eta)^{3d/2} \\ &\exp\left(\frac{|x|^2}{2\sigma^2}\left(\frac{p(1 + \eta)^2}{(1+\eta p)} - (1 + \eta)\right)\right)\\
    \leq \,&(c' \bar{k})^{\bar{k}}\eta^{-\bar{k}/2} \sigma^{-{\bar{k}}}(1 + \eta)^{3d/2}\\ &\exp\left(\frac{(p-1)|x|^2}{2\sigma^2}\left(1 + \eta p + \eta \right)\right).
\end{align*}
Substituting $\eta$ with $\eta/(p+1)$ and combining with the previous results, we establish the bound
\begin{align*}
    &|\partial^k f_\sigma(x)| \leq\\
    &(C')^d \bar{k}^{\bar{k}-1}p^{3d/2} \sigma^{1-\bar{k}} \exp\left(\frac{(p-1)|x|^2}{\sigma^2}\right)
\end{align*}
for some absolute constant $C' > 0$ and any $k \in \N_0^d$, when $\sigma \leq 1$.
\end{proof}

Next, we present a useful lemma concerning empirical approximation for IPMs whose function classes are sufficiently well-behaved.

\begin{lemma}
\label{lem:Donsker}
Let $\cF \subset C^{\alpha}(\R^{d})$ be a function class where $\alpha$ is a positive integer with $\alpha > d/2$, and let $\{ \cX_{j} \}_{j=1}^{\infty}$ be a cover of $\R^{d}$ consisting of nonempty bounded convex sets with bounded diameter. 
Set $M_j = \sup_{f \in \cF} \| f \|_{C^{\alpha}(\cX_j)}$ with $\| f \|_{C^{\alpha}(\cX_j)} = \max_{\bar{k} \le \alpha} \sup_{x \in \mathrm{int}(\cX_j)} |\partial^{k}f(x)|$. 
If $\sum_{j=1}^{\infty} M_{j}\mu(\cX_j)^{1/2} < \infty$, then $\cF$ is $\mu$-Donsker and $\E\big[\|\hat{\mu}_n-\mu\|_{\infty,\cF}\big] \lesssim n^{-1/2}\sum_{j=1}^{\infty}M_j \mu(\cX_j)^{1/2}$ up to constants that depend only on $d, \alpha$, and $\sup_{j}\diam (\cX_j)$. 
\end{lemma}

\begin{proof}[Proof of Lemma \ref{lem:Donsker}]

The lemma follows from Theorem 1.1 in \cite{vanderVaart1996}. Let $I_{1} = \cX_1$ and $I_j = \cX_{j} \setminus \bigcup_{k=1}^{j-1}\cX_{k}$ for $j=2,3,\dots$. The collection $\{ I_j \}$ forms a partition of $\R^{d}$. Define $\cF_{\cX_{j}} = \{ f\mathds{1}_{\cX_{j}} : f \in \cF \}$ and $\cF_{I_{j}} = \{ f\mathds{1}_{I_{j}} : f \in \cF \}$. 
Let $F = \sum_{j} M_{j}1_{I_{j}}$, which gives an envelope for $\cF$.
Observe that 
\[\mu(F^2) = \sum_{j} M_j^2 \mu(I_j) \le \sum_j M_j^2 \mu(\cX_j) < \infty,
\]
which also ensures that $\cF \subset L^2 (\mu)$.

In view of the discussion before Corollary 2.1 in \cite{vanderVaart1996}, we see that each $\cF_{\cX_{j}}$ is $\mu$-Donsker (which implies that $\cF_{I_{j}}$ is $\mu$-Donsker as $\cF_{I_{j}}$ can be viewed as a subset of $\cF_{\cX_{j}}$) and 
\begin{align*}
\E[ \| \sqrt{n}(\hat{\mu}_n-\mu) \|_{\infty,\cF_{I_{j}}}] &\le \E[ \| \sqrt{n}(\hat{\mu}_n-\mu) \|_{\infty,\cF_{\cX_{j}}}] \\
&\lesssim M_j \mu(\cX_j)^{1/2}
\end{align*}
up to constants that depend only on $d, \alpha$, and $\sup_{j}\diam (\cX_j)$.
The RHS is summable over $j$ so that by Theorem 1.1 in \cite{vanderVaart1996}, $\cF$ is $\mu$-Donsker.
The bound on $\E[\|\hat{\mu}_n-\mu\|_{\infty,\cF}]$ follows by summing up bounds on $\E[\|\hat{\mu}_n-\mu\|_{\infty,\cF_{I_j}}]$.
\end{proof}

We are now in position to prove \cref{thm:limit-distribution}. 

\begin{proof}[Proof of \cref{thm:limit-distribution}]

Observe that
\begin{equation}
\big ((\hat{\mu}_n - \mu)*\cN_\sigma \big) (f) = (\hat{\mu}_n - \mu)(f*\varphi_\sigma).
\label{eq: convolution identity}
\end{equation}
and consider the function classes
\begin{align}
\label{eq:sobolev-ball}
\cF &=\big \{ f \in C_{0}^{\infty} :  \| f \|_{\dot{H}^{1,q}(\cN_{\sigma})} \le 1 \big\}\\ \cF*\varphi_{\sigma} &= \big \{ f*\varphi_{\sigma} : f \in \cF \big \}. 
\end{align}
We apply \cref{lem:Donsker} to show that the function class $\cF*\varphi_{\sigma}$ is $\mu$-Donsker, implying the limit described in the theorem statement. 
Since for any constant $a \in \R$ and any function $f \in \cF$, $(\hat{\mu}_n - \mu)(f * \varphi_\sigma) = (\hat{\mu}_n - \mu)\big((f-a)* \varphi_\sigma \big)$, we only have to verify the conditions of Lemma \ref{lem:Donsker} for $\cF*\varphi_{\sigma}$ with $\cF$ replaced by $\{ f - \Gauss(f) : f \in C_{0}^{\infty}, \| f \|_{\dot{H}^{1,q}(\cN_{\sigma})} \le 1 \}$. 

We first construct a cover $\{ \cX_j \}_{j=1}^{\infty}$ as follows. Let $B_r = B(0,r)$. For $\delta > 0$ fixed and $r =2,3,\dots$, let $\{ x_{1}^{(r)},\dots,x_{N_{r}}^{(r)} \}$ be a minimal $\delta$-net of $B_{r\delta} \setminus B_{(r-1)\delta}$.
Set $x_{1}^{(1)} = 0$ with $N_1 = 1$. To bound $N_r$, we show that the covering number $N(B_{r\delta} \setminus B_{(r-1)\delta}, | \cdot |, \epsilon)$, defined as the size of the smallest $\epsilon$-cover of $B_{r\delta} \setminus B_{(r-1)\delta}$, satisfies
\begin{equation}
N(B_{r\delta} \setminus B_{(r-1)\delta}, | \cdot |, \epsilon) \le \left ( \frac{2r\delta}{\epsilon}+1 \right)^{d} \mspace{-3mu}-\mspace{-3mu} \left ( \frac{2(r\mspace{-3mu}-\mspace{-3mu}1)\delta}{\epsilon} \mspace{-3mu}-\mspace{-3mu}1  \right)^{d}\label{EQ:Nr_bound}
\end{equation}
for $0 < \epsilon \le 2(r-1)\delta$, according to a volumetric argument. 
Specifically, let $\{ x_1,\dots,x_N \}$ be a maximal $\epsilon$-separated subset of $B_{r\delta} \setminus B_{(r-1)\delta}$.
By maximality, $\{ x_1,\dots,x_N \}$ is an $\epsilon$-net of $B_{r\delta} \setminus B_{(r-1)\delta}$. 
By construction, 
\[
\bigcup_{j=1}^{N}B(x_j,\epsilon/2) \subset B_{r\delta+\epsilon/2} \setminus B_{(r-1)\delta-\epsilon/2}
\]
and the balls of the left-hand side (LHS) are disjoint. Comparing volumes, we have 
\[
N (\epsilon/2)^{d} \le (r \delta+\epsilon/2)^{d} - ((r-1)\delta-\epsilon/2)^{d}.
\]
This yields the bound on the covering number.

Given \eqref{EQ:Nr_bound}, we obtain $N_{r} \le (2r+1)^d-(2r-3)^d = O(r^{d-1})$. Set 
\[
\cX_j = B(x_j^{(r)},\delta), \ j=\sum_{k=1}^{r-1}N_{k}+1,\dots,\sum_{k=1}^{r}N_{k}.
\]
By construction, $\{ \cX_j \}_{j=1}^{\infty}$ forms a cover of $\R^{d}$ with diameter $2\delta$.
Set $\alpha = \lfloor d/2 \rfloor +1$ and $M_j = \sup_{f \in \cF:\, \cN_{\sigma}(f) = 0} \| f*\varphi_{\sigma} \|_{C^{\alpha}(\cX_{j})}$.
Fix any $\eta > 0$.
By \cref{lem:derivative}, 
\begin{align*}
&\max_{\sum_{k=1}^{r-1}N_{k}+1 \le j \le \sum_{k=1}^{r} N_{j}} M_j\\
\lesssim & \:\sigma^{-\lfloor d/2 \rfloor}  \exp \left ( \frac{(1 + \eta)(p-1)r^2 \delta^2}{2\sigma^2} \right )
\end{align*}
up to constants independent of $r$ and $\sigma$. 
Hence, in view of Lemma \ref{lem:Donsker}, the $\mu$-Donsker property of $\cF*\varphi_\sigma$ holds if 
\begin{align*}
\sum_{r=1}^{\infty} r^{d-1}\exp \left ( \frac{(1 + \eta)(p-1)r^2 \delta^2}{2\sigma^2} \right )\hspace{-1mm} \sqrt{\Prob(|X| > (r-1)\delta)}
\end{align*}
is finite. By Riemann approximation, the sum above can be bounded by $\delta^{-d-1}$ times
\[
 \int_{1}^{\infty}t^{d-1}\exp \left ( \frac{(1 + \eta)(p-1)t^2}{2\sigma^2 } \right ) \vspace{-1mm}\sqrt{\Prob(|X| > t-2\delta)} \dd t
\]
which is finite under our assumption by choosing $\eta$ and $\delta$ sufficiently small, and absorbing $t^{d-1}$ by the exponential term.

For more precise constants, we assume that $\mu$ is contained in a ball of radius $R$ centered at the origin. Then, using the constants from the proof of \cref{lem:derivative} with $\eta = 1$ and taking $\delta \leq R/2$, we find that the $\sqrt{n}\E\big[\Gds(\hat{\mu}_n,\mu)\big]$ is bounded by
\begin{align*}
&(C')^d d^{d/2}p^{3d/2} \sigma^{1-\bar{k}} 4^{d-1} \exp\left(\frac{4(p-1)R^2}{\sigma^2}\right)\\
    \leq & \, (cdp^3\sigma^{-1})^{d/2}e^{pR^2\sigma^{-2}},
\end{align*}
for some absolute constant $c > 0$, so long as $\sigma \leq 1$, say.
\end{proof}

\subsection{Proof of \cref{cor:fast-rate}}

The moment convergence of $\sqrt{n}\Gds(\hat{\mu}_n,\mu)$ follows from Lemma 2.3.11 in \cite{vandervaart1996book}. Finiteness of $\E[\| G \|_{\dot{H}^{-1,p}(\cN_\sigma)}]$ follows from Proposition A.2.3 in \cite{vandervaart1996book}. The second result follows from Theorem 1 after centering $\mu$ and $\hat{\mu}_n$ by the mean of $\mu$. Plugging in the constant from the previous proof, we find that
\[
\sqrt{n}\E\big[\GWp(\hat{\mu}_n,\mu)] \leq (cdp^3\sigma^{-1})^{d/2}e^{pR^2\sigma^{-2}}
\]
when $\mu$ is contained in a ball of radius $R$ and $\sigma \leq 1$, for some (different) constant $c > 0$.

\qed

\subsection{Proof of \cref{lem:subGaussian}}
Without loss of generality, we may assume that $X$ has mean zero. 
If $X$ is $\beta$-sub-Gaussian, then 
\[
\E[e^{\eta |X|^2}] \le \underbrace{(1-2\beta^2\eta)^{-d/2}}_{=C_\eta} \quad \text{if $\eta < 1/(2\beta^2)$}.
\]
By Markov's inequality, we have 
\[
\Prob (|X| > r) \le C_{\eta} e^{-\eta r^2}.
\]
Thus, 
\[
\int_{0}^{\infty} e^{\frac{\theta r^2}{2\sigma^2}} \sqrt{\Prob(|X|>r)} \dd r \le C_{\eta}^{1/2} \int_{0}^{\infty} e^{-\left ( \eta - \frac{\theta}{\sigma^2} \right) \frac{r^2}{2}} \dd r. 
\]
The right hand side is finite if and only if $\eta > \frac{\theta}{\sigma^2}$.
Such $\eta$ exists if and only if
\[
\frac{1}{2\beta^2} > \frac{\theta}{\sigma^2}, \quad \text{i.e.,} \quad \beta < \frac{\sigma}{\sqrt{2\theta}}. 
\]
Sine $\theta > p-1$ is arbitrary, we obtain the desired conclusion. \qed

\subsection{Proof of \cref{prop:concentration1}}
\label{prf:concentration1}
Given the comparison result of Theorem 1 and our characterization of $\E[\Gds(\hat{\mu}_n,\mu)]$ in the proof of Theorem 3, it suffices to prove
\begin{equation}
\Pr(\Gds(\hat{\mu}_n,\mu) \geq \E[\Gds(\hat{\mu}_n,\mu)] + t) \leq e^{c'nt^2}
\label{eq:Gds-concentration}
\end{equation}
for some constant $c' > 0$ independent of $n$ and $t$.
We apply Corollary 1 of \cite{Goldfeld2020GOT}, where the 1-Lipschitz function class $\mathsf{Lip}_1$ is substituted with $\cF_0 = \{ f - \Gauss(f): f \in C_0^\infty, \|f\|_{\dot{H}^{1,q}(\Gauss)} \leq 1 \}$. The desired conclusion follows according to the same argument, using McDiarmid's inequality, upon observing that for $x,x' \in \mathsf{supp}(\mu)$,
\begin{align*}
&\sup_{f \in \cF_0}\,  (f * \varphi_\sigma)(x) - (f * \varphi_\sigma)(x')\\
&\leq 2 \sup_{f \in \cF_0, y \in \mathsf{supp}(\mu)} (f * \varphi_\sigma)(y)\\
&\leq 2D_q(\Gauss)\exp\left(\frac{(p-1)R^2}{2\sigma^2}\right),
\end{align*}
where the final inequality uses \cref{lem:derivative}.  \qed

\section{Proofs for \cref{sec:computation}}
\label{prfs:computation}

First, we comment on a subtle detail regarding the construction of the homogeneous Sobolev space.

\begin{remark}%
\label{rem:explicit-construction}
For $\gamma \in \cP$ dominating the Lebesgue measure and satisfying the $p$-Poincar\'{e} inequality, the homogeneous Sobolev space $\dot{H}^{1,p}(\gamma)$ can be constructed as a function space over $\R^d$ that contains $\dot{C}_0^{\infty}$ as a dense subset in an explicit manner (without relying on the completion, which is an abstract metric-topological operation). See Appendix \ref{appen:Sobolev_space} for details of the construction. 
\end{remark}

Next, we observe that the inner product on $\dot{H}_0^{1,2}(\Gauss) * \varphi_\sigma$ is well-defined. That is, for $f,g \in \dot{H}_0^{1,2}(\Gauss)$, we show that $f * \varphi_\sigma = g * \varphi_\sigma$ if and only if $f = g$ almost everywhere. This requires an application of Wiener's Tauberian theorem for $L^2(\R^d)$, with a proof provided for completeness.

\begin{theorem}[Wiener's Tauberian theorem for $L^2$]
\label{thm:tauberian}
If the Fourier transform $\mathsf{F}[f]$ of $f \in L^2(\R^d)$ never vanishes, then the span of the set of translates $\{f_a : f_a(x) = f(a+x), a \in \R^d \}$ is dense in $L^2(\R^d)$.
\end{theorem}
\begin{proof}
Suppose that $g \in L^2(\R^d)$ is orthogonal to all translates of $f$. Then, because $\mathsf{F}$ is a unitary operator on $L^2(\R^d)$,
\begin{align*}
0 &= \int_{\R^d} g(x) f_a(x) \dd x\\
&= \int_{\R^d} \mathsf{F}[g](p) \mathsf{F}[f_a](p) \, \dd p\\
&= \int_{\R^d} e^{iap} \mathsf{F}[g](p) \mathsf{F}[f](p) \, \dd p
\end{align*}
for all $a \in \R^d$.
Equivalently, we have
\[
\mathsf{F}[\mathsf{F}[g]\cdot \mathsf{F}[f]](-a) = 0
\]
for all $a \in \R^d$. That is, $\mathsf{F}[\mathsf{F}[g]\cdot \mathsf{F}[f]] = 0$. Since $\mathsf{F}$ is injective, and $\mathsf{F}[f]$ never vanishes, we have $g = 0$, implying the desired density result.
\end{proof}

\begin{lemma}[Well-definedness of inner product]
For $f \in \ \dot{H}_0^{1,2}(\Gauss)$, $f * \varphi_\sigma = 0$ if and only if $f = 0$ almost everywhere.
\end{lemma}
\begin{proof}
By the previous remark, we can consider $f$ as an element of $L^2(\Gauss)$. The ``if'' direction is trivial. For the other direction, recall that $f$ can be realized as the limit in $L^2(\Gauss)$ of a sequence $\{f_n\}_{n \in \N}$ of simple functions with compact support. If $f * \varphi_\sigma = 0$, we have for any $y \in \R^d$ that
\begin{align*}
    &\left|\int_{\R^d} f_n(x) \varphi_\sigma (y - x) \,\dd x\right|\\
    = &\left|\int_{\R^d} (f_n - f)(x) \varphi_\sigma (y - x) \,\dd x\right|\\
    \leq & \|f - f_n\|_{L^2(\Gauss)} \int_{\R^d} \frac{\varphi_\sigma(y-x)^2}{\varphi_\sigma(x)} \dd x \to 0
\end{align*}
as $n \to \infty$. Because the Fourier transform of $\varphi_\sigma$ never vanishes, \cref{thm:tauberian} implies that the span of the functions $\varphi_\sigma(y - \cdot)$ is dense in $L^2(\R^d)$; thus, $ \langle f_n, g \rangle_{L^2(\R^d)} \to 0$ for any $g \in L^2(\R^d)$.
That is, the sequence $f_n$ converges weakly to 0 in $L^2(\R^d)$. Hence, $f_n$ must converge weakly to 0 in $L^2(\Gauss)$ as well (since the density $\varphi_\sigma$ is bounded). Seeing as $f$ is the ordinary limit of $f_n$ in $L^2(\Gauss)$, it must therefore coincide with the weak limit of 0.
\end{proof}

Now, since functions which are equal almost everywhere have the same convolution with $\varphi_\sigma$, this implies that $\dot{H}_0^{1,2}(\Gauss)*\varphi_\sigma$ is realizable as a Hilbert space of functions (not equivalence classes of functions), the most basic requirement for the RKHS property.

Next, we prove a lemma which allows us to concentrate on $\sigma = 1$ without loss of generality.

\begin{lemma}[Unit smoothing parameter]
\label{lem:sigma-reduction}
For $\mu,\nu \in \cP_p$, let $X \sim \mu$ and $Y \sim \nu$. Then,
\[
\Gds(\mu,\nu) = \sigma \, \mathsf{d}_p^{(1)}(\mu',\nu'),
\]
where $\mu'$ and $\nu'$ are the distributions of $X/\sigma$ and $Y/\sigma$, respectively.
\end{lemma}

\begin{proof}[Proof of Lemma \ref{lem:sigma-reduction}]
\label{prf:sigma-reduction}
First, define the isometric isomorphism $T:\dot{H}^{1,q}(\cN_\sigma) \to \dot{H}^{1,q}(\cN_1)$ by $(T f)(x) = \sigma^{-1}f(\sigma x)$. We verify
\begin{align*}
    \int_{\R^d} |\nabla_x (Tf)(x)|^q \, \dd \cN_1(x) &= \int_{\R^d} |\sigma^{-1} \nabla_x f(\sigma x)|^q \varphi_1(x)\, \dd x\\
    &= \int_{\R^d} |\nabla f(\sigma x)|^q \varphi_1(x)\, \dd x \\
    &= \int_{\R^d}|\nabla f(u)|^q \, \dd \cN_\sigma(u).
\end{align*}
Taking independent $X \sim \mu, Y \sim \nu$, and $Z \sim \cN_1$ and noting that $f(x) = \sigma\, Tf(x/\sigma)$, we have
\begin{align*}
(\mu-\nu)(f*\varphi_\sigma) &= \E \left[f(X + \sigma Z) - f(Y + \sigma Z)\right]\\
&= \sigma \cdot \E \left[Tf(X/\sigma + Z) - Tf(Y/\sigma + Z)\right]\\
&= \sigma \cdot(\mu' - \nu')(Tf * \varphi_1),
\end{align*}
where $\mu'$ and $\nu'$ are the distributions of $X/\sigma$ and $Y/\sigma$, respectively. Thus,
\begin{equation*}
\begin{split}
\Gds(\mu,\nu) &= \sup_{\substack{f \in \dot{H}^{1,q}(\cN_\sigma)\\ \|f\|_{\dot{H}^{1,q}(\cN_\sigma)} \leq 1}} (\mu - \nu)(f*\varphi_\sigma) \\
&= \sigma \sup_{\substack{Tf \in \dot{H}^{1,q}(\cN_1)\\ \|Tf\|_{\dot{H}^{1,q}(\cN_1)} \leq 1}} (\mu' - \nu')(Tf*\varphi_1) \\
&= \sigma\, \mathsf{d}_p^{(1)}(\mu',\nu').
\end{split}
\end{equation*}
This completes the proof.
\end{proof}

Next, we identify an orthonormal basis of $\dot{H}_0^{1,2}(\cN_1) * \varphi_1$.
We first prove that Hermite polynomials form an orthonormal basis of $\dot{H}_{0}^{1,2}(\cN_1)$, and then translate this to an orthonormal basis of $\dot{H}_0^{1,2}(\cN_1)*\varphi_1$.
Here, for $k \in \N_0^d$, we write $x^k := \prod_{i=1}^d x_i$ and $\bar{k} := \sum_{i=1}^d k_i$.

\begin{lemma}[Orthonormal basis of $\dot{H}_0^{1,2}(\cN_1) * \varphi_1$]
\label{lem:smoothed-Sobolev-ONB}
The monomials $\phi_k(x) = (\bar{k}\prod_{i=1}^d k_i)^{-1/2} x^k$, $0 \neq k \in \N_0^d$, comprise an orthonormal basis of $\dot{H}_0^{1,2}(\cN_1)*\varphi_1$.
\end{lemma}

\begin{proof}[Proof of Lemma \ref{lem:smoothed-Sobolev-ONB}]
Recall that the Hermite polynomials defined as
\[
\He_n(x) = (-1)^n e^{x^2/2} \dv[n]{x} e^{x^2/2}
\]
satisfy $\He_n'(x) = n\He_{n-1}(x)$ and $\int \He_n \He_m \, \dd \cN_1 = n! \, \delta_{n,m}$ \cite{bogachev_gaussian_1998}.
They admit a natural multivariate extension
\[
\He_k(x) = \prod_{i=1}^d \He_{k_i}(x_i), \quad k \in \N_0^d,
\]
which satisfies
\begin{align*}
\langle  \He_k, \He_{k'} \rangle_{\dot{H}^{1,2}(\cN_1)} &= \int \langle \nabla \He_k, \nabla \He_{k'} \rangle \dd \cN_1\\
&= \sum_{i=1}^d \int \pdv{\He_k}{x_i} \pdv{\He_{k'}}{x_i} \dd \cN_1\\
&= \delta_{k,k'}\bar{k} \prod_{i=1}^d k_i!.
\end{align*}
Thus, the normalized polynomials $\widetilde{\He}_k := \left(\bar{k}\prod k_i!\right)^{-1/2} \He_k$, $0 \neq k \in \N_0^d$, form an orthonormal set, and it is easy to check that they span the space of $d$-variate polynomials $Q$ with $\cN_1(Q) = 0$.
By Proposition 1.3 of \cite{schmuland92}, polynomials are dense in the inhomogeneous Gaussian Sobolev space $H^{1,2}(\cN_1)$, and hence $\dot{H}^{1,2}(\cN_1)$, so it follows that the $\widetilde{\He}_k$ polynomials form an orthonormal basis for $\dot{H}_0^{1,2}(\cN_1)$.

Next, we observe that, in one dimension, $(\He_n * \varphi_1)(x) = x^n$. To see this, we use the Rodrigues formula for the Hermite polynomials \cite{rasala1981}, which states that $\He_n(x) = e^{-D^2/2}[x^n]$.
Here, $D$ is the differentiation operator and $\exp$ is defined on operators via its formal power series (working with polynomials, there are no issues of convergence). 
We can express convolution with a standard Gaussian in a similar way, with $f*\varphi_1 = e^{D^2/2} f$ (where it suffices to consider only $f$ that are polynomials) \cite{bilodeau1962}. Together, these reveal that $(\He_n * \varphi_1)(x) = x^n$. Thus, for $0 \neq k \in \N_0^d$, we obtain
\[
(\widetilde{\He}_k*\varphi_1)(x) = \left(\bar{k}\prod k_i!\right)^{-1/2} x^k =: \phi_k(x).
\]
Since the $\widetilde{\He}_k$ polynomials form an orthonormal basis for $\dot{H}_0^{1,2}(\cN_1)$, the $\phi_k$ monomials form an orthonormal basis for $\dot{H}_0^{1,2}(\cN_1) * \varphi_1$, as claimed.
\end{proof}

Now, the theorem follows via routine calculations.

\subsection{Proof of \cref{thm:dual-Sobolev-RKHS}}
\label{prf:dual-Sobolev-RKHS}

By \cref{lem:derivative}, we have  that for any $f \in \dot{H}_0^{1,2}(\Gauss)$, 
\begin{align*}
&|(f*\varphi_\sigma)(x)| \leq  D_2(\cN_{\sigma}) e^{|x|^2/(2\sigma^2)} \| \nabla f\|_{L^2(\Gauss)}\\
&= e^{|x|^2/(2\sigma^2)} \|f*\varphi_\sigma\|_{\dot{H}^{1,2}(\Gauss) * \varphi_\sigma},
\end{align*}
so pointwise evaluation at $x$ is a bounded linear operator on $\dot{H}_0^{1,2}(\Gauss) * \varphi_\sigma$ for each $x \in \R^d$.
This implies that $\dot{H}_0^{1,2}(\Gauss) * \varphi_\sigma$ is an RKHS over $\R^d$.
For $\sigma = 1$, we can compute the reproducing kernel from the orthonormal basis above (see Theorem 4.20 of \cite{steinwart2008}) as
\begin{align*}
\kappa^{(1)}(x,y) &= \sum_{0 \neq k \in \N_0^d} \phi_k(x) \phi_k(y)\\
&= \sum_{0 \neq k \in \N_0^d}\left(|k|\prod k_i!\right)^{-1} x^k y^k \\
&= \sum_{n = 1}^{\infty} \frac{1}{n\cdot n!}\sum_{|k| = n}\frac{n!}{\prod k_i!} x^ky^k\\
&= \sum_{n = 1}^{\infty} \frac{1}{n\cdot n!}\langle x,y \rangle^n = -\Ein(-\langle x,y \rangle).
\end{align*}
We note that $\kappa^{(1)}$ is positive semi-definite by this construction.
The MMD formulation \eqref{eq:MMD}
follows because
\begin{align*}
&\mathsf{d}_2^{(1)}(\mu,\nu)\\
&=  \sup \Big \{ \mu(f*\varphi_1) - \nu(f*\varphi_1) :\\
&\mspace{68mu} f \in \dot{H}_0^{1,2}(\cN_1), \| f \|_{\dot{H}^{1,2}(\cN_1)} \le 1 \Big \}\\
&=  \sup \Big \{ \mu(g) - \nu(g) :\\
&\mspace{68mu} g \in \dot{H}_0^{1,2}(\cN_1)*{\varphi_1}, \| g \|_{\dot{H}^{1,2}(\cN_1)*\varphi_1} \le 1 \Big \}\\
&= \mathsf{MMD}_{\dot{H}_0^{1,2}(\cN_1)*{\varphi_1}}(\mu,\nu)
\end{align*}
and the RHS of (5) %
is the standard kernel formulation of an MMD \cite{gretton_two_sample_2012}.
The extension to general $\sigma$ follows from \cref{lem:sigma-reduction} and the uniqueness of the reproducing kernel.\qed

\section{Proofs for \cref{sec:applications}}
\label{prf:applications}

\subsection{Proof of \cref{prop:two-sample_SW}}
\label{prf:two-sample_SW}
We first consider the size control. Suppose that $\mu = \nu$. Without loss of generality, we may assume that $\mu$ is not a point mass. To handle shifts of distributions, for any $a \in \R^{d}$, we represent
\begin{align*}
&\sqrt{\frac{mn}{N}} \Gds(\hat{\mu}_m*\delta_{-a},\hat{\nu}_n*\delta_{-a}) \\
&= \Bigg \|  \sqrt{\frac{n}{N}} \sqrt{m}(\hat{\mu}_m-\mu)(f(\cdot-a)*\varphi_{\sigma})\\
&\mspace{30mu}- \sqrt{\frac{m}{N}} \sqrt{n}(\hat{\nu}_n-\mu)(f(\cdot-a)*\varphi_{\sigma}) \Bigg \|_{\infty,\cF},
\end{align*}
where the function class $\cF = \{ f \in C_{0}^{\infty} : \| f \|_{\dot{H}^{1,q}(\cN_{\sigma})} \le 1 \}$ is the one from the proof of Theorem 3. %
Consider another function class 
\[
\cF_{\mathsf{shift}} = \{ f(\cdot - a) : f \in C_{0}^{\infty}, \| f \|_{\dot{H}^{1,q}(\cN_{\sigma})} \le 1, |a| \le C \}
\]
for some large enough constant $C < \infty$ such that the mean $a_\mu$ of $\mu$ satisfies $|a_\mu| < C$. It is not difficult to see from the proof of Theorem 3 %
that the function class $\cF_{\mathsf{shift}}*\varphi_{\sigma}$ is $\mu$-Donsker, which implies that (cf. Theorem 1.5.7 in \cite{vandervaart1996book})
\begin{align*}
\limsup_{m \to \infty}\Prob \Bigg ( \sup_{\substack{f \in C_{0}^{\infty}\\\| f \|_{\dot{H}^{1,q}(\cN_{\sigma})} \le 1 \\ |a-b| < \delta}}\big| \sqrt{m}(\hat{\mu}_m - \mu)\big ((f(\cdot-a)\\
-f(\cdot-b))*\varphi_\sigma \big) \big| > \epsilon \Bigg) \to 0
\end{align*}
as $\delta \to 0$, for all $\epsilon > 0$.
Here we used the fact that $|a-b| \to 0$ implies that $\Var \big((f(X-a)-f(X-b))*\varphi_\sigma\big) \to 0$. Since $\bar{X}_{m} \to a_{\mu}$ a.s. by the law of large numbers, we have
\begin{align*}
\sup_{\substack{f \in C_{0}^{\infty}\\ \| f \|_{\dot{H}^{1,q}(\cN_{\sigma})} \le 1}}\big| \sqrt{m}(\hat{\mu}_m - \mu)\big ((f(\cdot-\bar{X}_m)\\
-f(\cdot-a_{\mu}))*\varphi_\sigma \big) \big| \stackrel{\Prob}{\to} 0.
\end{align*}
A similar result holds for $\hat{\nu}_n$. 
Now, by Theorem 3, we have
\begin{equation}
\begin{split}
W_{m,n} &\le p \sqrt{\frac{mn}{N}}\\
&\min \Big \{ e^{\tr \hat{\Sigma}_{X}/(2q \sigma^2)}\Gds(\hat{\mu}_m*\delta_{-\bar{X}_m},\hat{\nu}_n*\delta_{-\bar{X}_m}),\\
&\mspace{50mu} e^{\tr \hat{\Sigma}_{Y}/(2q \sigma^2)}\Gds(\hat{\mu}_m*\delta_{-\bar{Y}_n},\hat{\nu}_n*\delta_{-\bar{Y}_n})\Big \},
\end{split}
\end{equation}

In view of this inequality, together with the fact that $\hat{\Sigma}_{X} \to \Sigma_{\mu}$ and $\hat{\Sigma}_{Y} \to \Sigma_{\mu}$ a.s., where $\Sigma_\mu$ is the covariance matrix of $\mu$, we conclude that $W_{m,n}$ is at most
\begin{align*}
p e^{\tr \Sigma_{\mu}/(2q\sigma^2)} \sqrt{\frac{mn}{N}} \Gds(\hat{\mu}_m*\delta_{-a_\mu},\hat{\nu}_n*\delta_{-a_\mu}) + o_{\Prob}(1). 
\end{align*}
Now, the function class $\cF*\varphi_{\sigma}$ is Donsker w.r.t. $\mu*\delta_{-a_{\mu}}$, so that from p. 361 of \cite{vandervaart1996book}, we have 
\[
\sqrt{\frac{mn}{N}} \Gds(\hat{\mu}_m*\delta_{-a_\mu},\hat{\nu}_n*\delta_{-a_\mu}) \stackrel{d}{\to} \|G\|_{\cF,\infty}, 
\]
where $G$ is the Gaussian process that appears in Theorem 3 with $\mu$ replaced by $\mu*\delta_{-a_{\mu}}$.

Is is easy to show that the distribution function of $\|G\|_{\dot{H}^{-1,p}(\cN_{\sigma})}$ is continuous (cf. the proof of Lemma 3 in \cite{goldfeld2020asymptotic}), so long as $\mu$ is not a point mass (in which case the proposition is trivially true). To show that the test has asymptotic level $\alpha$, it then suffices to show that (cf. Lemma 23.3 in \cite{vandervaart1998}) 
\begin{equation}
\begin{split}
\Prob^{B} \left(\sqrt{\frac{mn}{N}}\Gds\left(\hat{\mu}_{m}^B*\delta_{-\bar{Z}_N},\hat{\nu}_{n}^B*\delta_{-\bar{Z}_N}\right) \le t \right )\\
\stackrel{\Prob}{\to} \Prob \left (\|G\|_{\dot{H}^{-1,p}(\cN_{\sigma})} \le t \right), \quad \forall t \ge 0. 
\end{split}
\label{eq:bootstrap_consistency}
\end{equation}
Observe that 
\[
\begin{split}
&\sqrt{\frac{mn}{N}}\Gds\left(\hat{\mu}_{m}^B*\delta_{-\bar{Z}_N},\hat{\nu}_{n}^B*\delta_{-\bar{Z}_N}\right) \\
&= \Bigg \|  \sqrt{\frac{n}{N}} \sqrt{m}\big(\hat{\mu}_m^B-\hat{\gamma}_N\big)\big(f(\cdot-\bar{Z}_N)*\varphi_{\sigma}\big) \\
&\mspace{30mu}- \sqrt{\frac{m}{N}} \sqrt{n}\big(\hat{\nu}_n^B-\hat{\gamma}_N\big)\big(f(\cdot-\bar{Z}_N)*\varphi_{\sigma}\big) \Bigg \|_{\infty,\cF}.
\end{split}
\]
Since the function class $\cF_{\mathsf{shift}}*\varphi_{\sigma}$ is $\mu$-Donsker, 
by Theorem 3.6.1 in \cite{vandervaart1996book}, the bootstrap process $\sqrt{m}(\hat{\mu}_m^B - \hat{\gamma}_N)$ indexed by $\cF_{\mathsf{shift}} * \varphi_{\sigma}$ converges in distribution in $\ell^{\infty}(\cF_{\mathsf{shift}}*\varphi_{\sigma})$ \textit{unconditionally}, which implies that
\begin{align*}
\limsup_{m,n \to \infty}\Prob \Bigg ( \sup_{\substack{f \in C_{0}^{\infty}\\ \| f \|_{\dot{H}^{1,q}(\cN_{\sigma}) \le 1} \\ |a-b| < \delta}}\big| \sqrt{m}\big(\hat{\mu}_m^B - \hat{\gamma}_N\big)\big ((f(\cdot-a)\\
-f(\cdot-b))*\varphi_\sigma \big) \big| > \epsilon \Bigg) \to 0
\end{align*}
as $\delta \to 0$, for all $\epsilon > 0$.
Since $\bar{Z}_{N} \to a_{\mu}$ a.s. by the law of large numbers, we have
\begin{align*}
\sup_{f \in C_{0}^{\infty}, \| f \|_{\dot{H}^{1,q}(\cN_{\sigma})} \le 1}\big| \sqrt{m}\big(\hat{\mu}_m^B - \hat{\gamma}_N\big)\big ((f(\cdot-\bar{Z}_N)\\
-f(\cdot-a_{\mu}))*\varphi_\sigma \big) \big| \stackrel{\Prob}{\to} 0.
\end{align*}
An analogous result holds for $\hat{\nu}_{n}^B$. Thus, we have 
\begin{align*}
&\sqrt{\frac{mn}{N}}\Gds(\hat{\mu}_{m}^B*\delta_{-\bar{Z}_N},\hat{\nu}_{n}^B*\delta_{-\bar{Z}_N})\\
&= \sqrt{\frac{mn}{N}}\Gds(\hat{\mu}_{m}^B*\delta_{-a_{\mu}},\hat{\nu}_{n}^B*\delta_{-a_{\mu}})  + o_{\Prob}(1).
\end{align*}
The desired conclusion (\ref{eq:bootstrap_consistency}) follows from Theorem 3.7.6 in \cite{vandervaart1996book} combined with the fact that the function class $\cF*\gauss$ is $\mu*\delta_{-a_{\mu}}$-Donsker.

To show asymptotic consistency, suppose that $\mu \ne \nu$ and note that the preceding argument and Theorem 3.7.6 in \cite{vandervaart1996book} imply that 
\[
\Prob^{B}\left(W_{m,n}^B \le t\right) \stackrel{\Prob}{\to} \Prob \left ( pe^{\tr \Sigma_{\gamma}/(2q\sigma^2)} \| G_{\gamma} \|_{\dot{H}^{-1,p}(\cN_\sigma)} \le t \right)
\]
for all $t \geq 0$, where $\Sigma_{\gamma}$ is the covariance matrix of the measure $\gamma = \tau \mu + (1-\tau) \nu$ and $G_{\gamma}$ is the Gaussian process from Theorem 3 with $\mu$ replaced by $\gamma*\delta_{-a_{\gamma}}$ ($a_{\gamma}$ is the mean vector of $\gamma$). Furthermore, it is not difficult to see that $W_{m,n} \stackrel{\Prob}{\to} \infty$ under the alternative, which implies that $\Prob\big(W_{m,n} > w_{m,n}^{B}(1-\alpha)\big) \to 1$ whenever $\mu \ne \nu$. 
\qed

Propositions 7,8, and 9 follow from essentially similar proofs to those in \cite{goldfeld2020asymptotic}, which build on \cite{bernton2019} and \cite{pollard_min_dist_1980}, with arbitrary $p\geq 1$ instead of the $p=1$ considered therein (indeed, the needed results from \cite{villani2008optimal} hold for all $1 \leq p < \infty$), so we omit their proofs for brevity. 

\subsection{Proof of \cref{cor:M-SWE-generalization-error}}
\label{prf:M-SWE-generalization-error}
First, we state a simple lemma to bound generalization error of minimum distance estimation w.r.t. an IPM in terms of the empirical approximation error.

\begin{lemma}[Generalization error for GANs]
\label{lem:generalization-error}
For an IPM $\mathsf{d}$ and an estimator $\hat{\theta}_n\in\Theta$ with $\mathsf{d}(\hat{\mu}_n,\nu_{\hat{\theta}_n}) \leq \inf_{\theta \in \Theta} \mathsf{d}(\hat{\mu}_n,\nu_\theta) + \epsilon$, we have
\[
\mathsf{d}(\mu,\nu_{\hat{\theta}_n}) - \inf_{\theta \in \Theta}\mathsf{d}(\mu,\nu_\theta) \leq 2\, \mathsf{d}(\mu,\hat{\mu}_n) + \epsilon.
\]
\end{lemma}

This is a consequence of the triangle inequality, see \cite{zhang_generalization_2018} for example. Hence, our conclusion follows upon noting that
\begin{align*}
\Prob\left(2\GWp(\mu,\hat{\mu}_n) > t \right) &\leq \Prob\left(\Gds(\mu,\hat{\mu}_n) > Ct \right)\\
&\leq \exp(-n(Ct - C'n^{-1/2})^2)\\
&\leq C_1\exp(-C_2 nt^2),
\end{align*}
where constants $C,C',C_1,C_2$ are independent of $n$ and $t$. Here, we have combined the concentration result \eqref{eq:Gds-concentration}, the comparison from Theorem 1, and the fast rate from Corollary 2.\qed

\section{Additional Details for Experiments}
\label{app:experiments}

In Figure \ref{fig:S-MWE-2}, we present additional S-MWE experiments for a single Gaussian parameterized by mean and variance, demonstrating similar limiting behavior to the mixture results provided in the main text.

We note that experiments for Figures \ref{fig:empirical-convergence}, \ref{fig:S-MWE}, and \ref{fig:S-MWE-2} were performed on a Dell OptiPlex 7050 PC with 32GB RAM and an 8 core 2.80GHz Intel Core i7 CPU, running in approximately 3 hours, 30 minutes, and 30 minutes, respectively. Computations for Figure \ref{fig:S-MWE} were performed on a cluster instance with 14 vCPUs and 112 GB RAM over several hours. Those for Figure \ref{fig:two-sample} were performed on a cluster machine with 14 vCPUs, 60 GB RAM, and a NVIDIA Tesla V100 over nearly 12 hours (hence the restriction to low dimensions).

\begin{figure}[h]
\includegraphics[width=8cm]{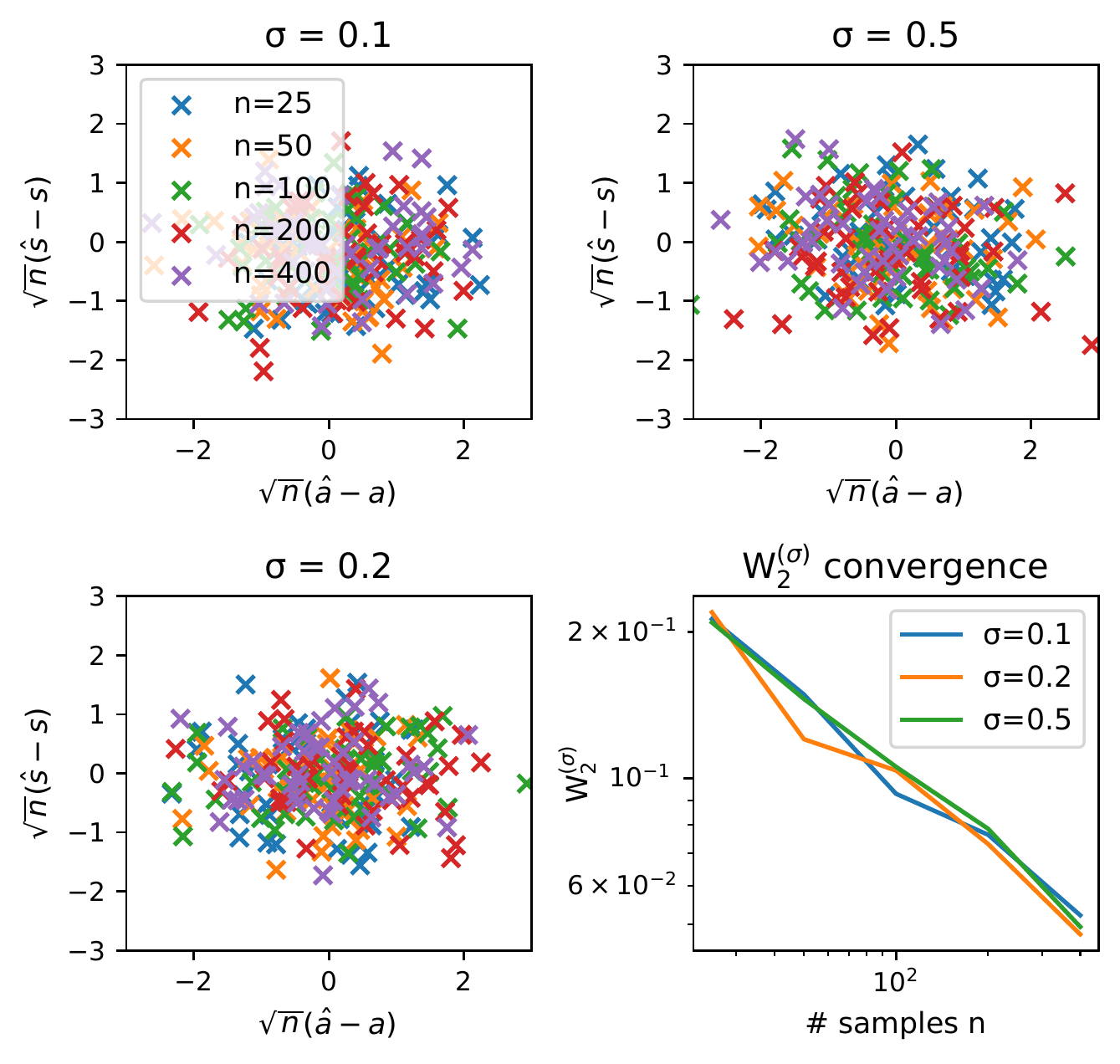}
\vspace{-1em}
\caption{One-dimensional limiting behavior of M-SWE estimates for the mean and standard deviation parameters of $\mu = \cN(a,s)$ with $a = 0$ and $s = 1$. Also shown is a log-log plot of $\mathsf{W}_2^{(\sigma)}$ convergence in $n$.}
\label{fig:S-MWE-2}
\end{figure}

Finally, we describe how the upper bound on $\mathsf{W}_2^{(\sigma)}$ was computed for the rightmost plot of Figure \ref{fig:empirical-convergence}.

\subsection{Upper Bound on $\E\big[\mathsf{W}_2^{(\sigma)}(\hat{\mu}_n,\mu)\big]$ using $\mathsf{d}_2^{(\sigma)}$}
\label{app:rough-upper-bound}

By \cref{thm:dual-Sobolev-RKHS}, we have
\begin{align*}
    \mathsf{d}_2^{(\sigma)}(\hat{\mu}_n,\mu)^2 =
    \E\big[\kappa^{(\sigma)}(X,X')\big] &+ \frac{1}{n^2}\sum_{i,j=1}^n \kappa^{(\sigma)}(X_i,X_j)\\ &- \frac{2}{n}\sum_{i=1}^n \E\big[\kappa^{(\sigma)}(X,X_i)\big],
\end{align*}
where $X,X' \sim \mu$ are independent. Taking expectations, we obtain
\begin{align*}
    &\E\Big[\mathsf{d}_2^{(\sigma)}(\hat{\mu}_n,\mu)^2\Big] =
    \E\big[\kappa^{(\sigma)}(X,X')\big] + 
    \frac{1}{n}\E\big[\kappa^{(\sigma)}(X,X)\big]\\ &+
    \left(1 - \frac{1}{n}\right) \E\big[\kappa^{(\sigma)}(X,X')\big] - \frac{2}{n}\sum_{i=1}^n \E\big[\kappa^{(\sigma)}(X,X_i)\big]\\
    &= \frac{1}{n}\left(\E\big[\kappa^{(\sigma)}(X,X)\big] - \E\big[\kappa^{(\sigma)}(X,X')\big]\right).
\end{align*}

Combining this with \cref{thm:smooth-Wp-Sobolev-comparison}, we reach the upper bound
\begin{align*}
    &\E\left[\mathsf{W}_2^{(\sigma)}(\hat{\mu}_n,\mu)\right] \leq 2 e^{\E[|X|^2]/(4\sigma^2)}\\
    &\left(\E\big[\kappa^{(\sigma)}(X,X)\big] - \E\big[\kappa^{(\sigma)}(X,X')\big]\right)^{1/2} n^{-1/2}.
\end{align*}
For Figure \ref{fig:empirical-convergence}, we estimate the kernel expectations via Monte Carlo integration with 1,000,000 samples. The kernel itself is computed via standard series-based methods for exponential integrals.

\section{Explicit Construction of the Homogeneous Sobolev space}
\label{appen:Sobolev_space}
Let $\gamma \in \cP$ be 
dominating the Lebesgue measure and satisfying the $p$-Poincar\'{e} inequality. Consider the homogeneous Sobolev space $\dot{H}^{1,p}(\gamma)$, which is constructed in %
Section 2 as the completion of $\dot{C}_{0}^\infty$ w.r.t. $\| \cdot \|_{\dot{H}^{1,p}(\gamma)}$. As such, it is not clear that the obtained space is a function space over $\R^d$. To show this is nevertheless the case, we present an explicit construction of $\dot{H}^{1,p}(\gamma)$ that does not rely on the completion.

Let $\cC=\{ f \in \dot{C}_{0}^{\infty} : \gamma(f) = 0 \}$. Then, $\| \cdot \|_{\dot{H}^{1,p}(\gamma)}$ is a proper norm on $\cC$, and the map $\iota: f \mapsto \nabla f$ is an isometry from $(\cC,\| \cdot \|_{\dot{H}^{1,p}(\gamma)})$ into $(L^{p}(\gamma;\R^d),\| \cdot \|_{L^{p}(\gamma;\R^d)})$. Let $V$ be the closure of $\iota \cC$ in $L^{p}(\gamma;\R^d)$ under $\| \cdot \|_{L^{p}(\gamma;\R^d)}$. The inverse map $\iota^{-1}: \iota \cC \to \cC$ can be extended to $V$. Indeed, for any $g \in V$, choose $f_n \in \cC$ such that $\| \nabla f_n - g \|_{L^p(\gamma;\R^d)} \to 0$. Since $\nabla f_n$ is Cauchy in $L^p(\gamma;\R^d)$, $f_n$ is Cauchy in $L^p(\gamma)$ by the $p$-Poincar\'{e} inequality, so $\| f_n - f \|_{L^p(\gamma)} \to 0$ for some $f \in L^p(\gamma)$. Set $\iota^{-1} g = f$ and extend $\| \cdot \|_{\dot{H}^{1,p}(\gamma)}$ by $\| f \|_{\dot{H}^{1,p}(\gamma)} = \lim_{n \to \infty} \| f_n \|_{\dot{H}^{1,p}(\gamma)}$.  The space $(\iota^{-1}V,\|\cdot \|_{\dot{H}^{1,p}(\gamma)})$ is a Banach space of functions over $\R^d$.

The homogeneous Sobolev space $\dot{H}^{1,p}(\gamma)$ is now constructed as $\dot{H}^{1,p}(\gamma) = \big\{ f+a : a \in \R, f \in \iota^{-1}V \big\}$ with $\| f+a \|_{\dot{H}^{1,p}(\gamma)} = \| f \|_{\dot{H}^{1,p}(\gamma)}$. 

\end{document}